\documentclass[11pt]{amsart}
\usepackage{verbatim,amssymb,amsmath,mathrsfs, graphicx,hyperref}
\usepackage{calc}
\usepackage{enumitem}
\usepackage[dvipsnames]{xcolor}
  \scrollmode

\newcommand{\R}{{\mathbb R}}
\newcommand{\N}{{\mathbb N}}

\newcommand{\EE}{{\mathbb E}}
\newcommand{\PP}{{\mathbb P}}

\newcommand{\Acal}{{\mathcal A}}

\newcommand{\sgn}{\operatorname{sgn}}
\newcommand{\eps}{\varepsilon}

\newcommand{\tX}{\widetilde X}

\newcommand{\tY}{\widetilde Y}
\newcommand{\tW}{\widetilde W}
\newcommand{\tV}{\widetilde V}
\newcommand{\tB}{\widetilde B}
\newcommand{\tmu}{\widetilde \mu}
\newcommand{\tsigma}{\widetilde \sigma}
\newcommand{\ttau}{\widetilde \tau}

\newcommand{\Fc}{\mathcal{F}}

\theoremstyle{plain}
\newtheorem{theorem}{Theorem}
\newtheorem{prop}{Proposition}
\newtheorem{lemma}{Lemma}
\newtheorem{cor}{Corollary}

\theoremstyle{definition}

\begin{document}

\title[Lower error bounds for approximation of SDEs with irregular coefficients]{The local coupling of noise technique \\ and its application to lower error bounds \\for strong approximation of SDEs \\ with irregular coefficients}

\author[Ellinger]
{Simon Ellinger}
\address{
Faculty of Computer Science and Mathematics\\
University of Passau\\
Innstrasse 33 \\
94032 Passau\\
Germany} \email{simon.ellinger@uni-passau.de}
\begin{abstract}
	In recent years, interest in approximation methods for stochastic differential equations (SDEs) with non-Lipschitz continuous coefficients has increased. We show lower bounds for the $L^p$-error of such methods in the case of approximation at a single point in time or globally in time. On the one hand, we show that for a large class of piecewise Lipschitz continuous drifts and non-additive diffusions the best possible $L^p$-error rate for final time approximation that can be achieved by any method based on finitely many evaluations of the driving Brownian motion is at most $3/4$, which was previously known only for additive diffusions. Moreover, we show that the best $L^p$-error rate for global approximation that can be achieved by any method based on finitely many evaluations of the driving Brownian motion is at most $1/2$ when the drift is locally bounded and the diffusion is locally Lipschitz continuous.
	
	For the derivation of the lower bounds we introduce a new method of proof: the local coupling of noise technique. Using this technique when approximating a solution $X$ of the SDE at the final time, a lower bound for the $L^p$-error of any approximation method based on evaluations of the driving Brownian motion at the points $t_1 < \dots < t_n$ can be determined by the $L^p$-distances of solutions of the same SDE on $[t_{i-1}, t_i]$ with initial values $X_{t_{i-1}}$ and driving Brownian motions that are coupled at $t_{i-1}, t_i$ and independent, conditioned on the values of the Brownian motion at $t_{i-1}, t_i$.
\end{abstract}
\maketitle
Consider a scalar autonomous stochastic differential equation (SDE)
\begin{equation}\label{sde0}
	\begin{aligned}
		dX_t & = \mu(X_t) \, dt +  \sigma(X_t) \, dW_t, \quad t\in [0,T],\\
		X_0 & = x_0
	\end{aligned}
\end{equation}
with deterministic initial value $x_0\in\R$, drift coefficient $\mu\colon\R\to\R$, diffusion coefficient $\sigma\colon \R \rightarrow \R$ and a one-dimensional driving
Brownian motion $W=(W_t)_{t\in[0,T]}$, where $T \in (0, \infty)$. Assume that the coefficients $\mu, \sigma$ are regular enough such that there exists a strong solution of the SDE~\eqref{sde0}.

In recent years, SDEs with irregular coefficients have gained increasing interest, where irregular means that the coefficients do not have to be Lipschitz continuous. It was shown in \cite{GJS17, hhj12, hjk11, JMGY15, MGRY2018, Y17} that the Euler scheme does not converge with any polynomial decay to the solution $X$ in general. Therefore, the question arose under which more general assumptions than Lipschitz continuity of the coefficients solutions of SDEs can be approximated well by the Euler or the Milstein scheme. In particular, increasing attention was paid to numerical methods for the approximation of SDEs with discontinuous drift, see \cite{GLN17, g98b, gk96b, HalidiasKloeden2008, LS16, LS15b, LS18, NSS19, Tag16, Tag2017b, Tag2017a}.

One special class of discontinuous drifts is the class of so-called piecewise Lipschitz continuous coefficients. This means that the coefficients are Lipschitz continuous on finitely many intervals and can have finitely many jumps. In \cite{LS16} it was shown that for piecewise Lipschitz continuous coefficients, for which
\begin{itemize}
	\item[(A1)] there exist a natural number $k \in \mathbb{N}$ as well as $-\infty = \xi_0 < \xi_1 < \dots < \xi_k < \xi_{k+1}=\infty$ such that $\mu$ is Lipschitz continuous on $(\xi_{i-1}, \xi_i)$ for all $i \in \{1, \dots, k+1\}$,
	\item[(A2)] $\sigma$ is Lipschitz continuous and $\sigma(\xi_i) \neq 0$ for all $i \in \{1, \dots, k\}$,
\end{itemize}
the equation~\eqref{sde0} has a unique strong solution. Convergence rates for strong approximation of such SDEs were investigated in~\cite{LS16, LS15b, LS18, MGY20, MGY19b, NSS19, Y2022}.
In \cite{MGY19b} it was proven that under the assumptions (A1),(A2),
\begin{itemize}
	\item[(A3)] $\sigma$ has a Lipschitz continuous derivative on $(\xi_{i-1}, \xi_i)$ for all $i \in \{1, \dots, k+1\}$,
	\item[(A4)] $\mu$ has a Lipschitz continuous derivative on $(\xi_{i-1}, \xi_i)$ for all $i \in \{1, \dots, k+1\}$,
\end{itemize}
a transformed Milstein scheme converges with a rate of at least $3/4$ in terms of the number of evaluations of $W$. For the additive case $\sigma = 1$, it was then shown in \cite{ELL24}, building on \cite{MGY23}, that the best possible $L^p$-error rate that can be achieved by any method based on finitely many evaluations of the driving Brownian motion $W$ is $3/4$ if the assumptions (A1)-(A4) hold and if there is a real jump position, i.e. there is an $i \in \{1, \dots, k\}$ with $\mu(\xi_i-) \ne \mu(\xi_i+)$. The proof of the statement is based on the global coupling of noise technique from \cite{MGY23}, where the upper bound for the transformed Milstein scheme is used to derive the lower bound. We introduce a new method of proof, namely the local coupling of noise technique, which can be used to derive lower bounds for many SDEs where no upper bounds for approximation errors need to be known. As an exemplary application of the technique, we show that one can drop (A4) and even $\sigma = 1$ and we still obtain that $3/4$ is the best possible $L^p$-error rate of methods based on finitely many evaluations of $W$ if there is a real jump position which is reached by the solution.

\begin{theorem}\label{cor:pwLip}
	Assume that (A1)-(A3) hold and let $T=1$.
	Let $x_0 \in \mathbb{R}$ and let $X\colon [0,1] \times \Omega \rightarrow \mathbb{R}$ be a strong solution of the SDE~\eqref{sde0} on the time interval $[0,1]$ with initial value $x_0$ and driving Brownian motion $W$ such that there is an $i \in \{1, \dots, k\}$ for which
	\begin{itemize}
		\item[($\alpha 1$)] it holds $\bigl(\frac{\mu}{\sigma} - \frac{\sigma'}{2}\bigr)(\xi_i+) \neq \bigl(\frac{\mu}{\sigma} - \frac{\sigma'}{2}\bigr)(\xi_i-)$, 
		\item[($\alpha 2$)] there exists a $t^\ast \in (0,1]$ such that for the local density $p_{X_{t^\ast}}$ of $X_{t^\ast}$ it holds $p_{X_{t^\ast}}(\xi_i) > 0$.
	\end{itemize}
	Then there exists a constant $c \in (0, \infty)$ such that for all $n \in \mathbb{N}$,
	\begin{flalign*}
		\inf_{\substack{t_1, \dots, t_n \in [0,1] \\g \colon \mathbb{R}^n \rightarrow \mathbb{R} \: measurable}} \Big( \EE \big[ |X_1 - g(W_{t_1}, \dots, W_{t_n})|^2 \big] \Big)^{1/2} \geq \frac{c}{n^{3 \slash 4}}.
	\end{flalign*}
\end{theorem}

We will show that the statement of Theorem~\ref{cor:pwLip} holds if (A1)-(A3) only hold on an interval around $\xi_i$ with $i$ from $(\alpha 1)$ and a transformation condition is fulfilled.

In addition to final time approximation, we also investigate global approximations using the local coupling of noise technique. It is well-known that the Euler scheme approximates the solution $X$ of a non-autonomous SDE in the global sense with a rate of at least $1/2$ in terms of the number of evaluations of $W$, provided that the coefficients are Lipschitz continuous. The following theorem shows that this rate is optimal.

\begin{theorem}\label{thm:nonadapt}
	Let $T \in (0, \infty)$, $\mu, \sigma\colon [0,T] \times \R \rightarrow \R$ be measurable functions and let $X\colon [0, T] \times \Omega \rightarrow \R$ be an adapted process with continuous paths such that 
	\begin{equation*}
		X_t = X_0 + \int_0^t \mu(s, X_s) \, ds + \int_0^t \sigma(s, X_s) \, dW_s, \qquad t \in [0,T].
	\end{equation*}
	Assume that there exist $t_0 \in [0, T), T_0 \in (t_0, T], \delta \in (0, \infty)$ and $\xi \in \R$ such that
	
	\begin{itemize}[align=left,labelwidth=\widthof{(local Lip)},leftmargin=\labelwidth+\labelsep]
		\item [(local Lip)] $\mu, \sigma$ are Lipschitz continuous on $[t_0, T_0] \times [\xi-\delta, \xi+\delta]$,
		\item [(non-deg)] $\inf_{(t,x) \in [t_0, T_0] \times B_\delta(\xi)} |\sigma(t,x)| > 0$,
		\item [(reach)] $\PP(X_{t_0} \in B_\delta(\xi)) > 0$.
	\end{itemize}
	
	Then there exists a constant $c \in (0, \infty)$ such that for all $n \in \mathbb{N}$,
	\[
	\inf_{\substack{t_1, \dots, t_n \in [0,T] \\g \colon \mathbb{R}^n \rightarrow L^1([0,T]) \: measurable}} \EE\big[ \|X - g(W_{t_1}, \dots, W_{t_n}) \|_{L^1([0,T])}\big] \ge \frac{c}{n^{1/2}}.
	\]
\end{theorem}

The above theorem thus extends the corresponding statement of Theorem 12 in~\cite{hhmg2019}, where $\mu$ and $\sigma$ must have continuous first order partial derivatives in the time variable and continuous second order partial derivatives in the state variable, and fits better to the Lipschitz continuity assumptions for the Euler scheme. As in~\cite{hhmg2019}, we also show that the rate 1/2 is optimal even for adaptive methods that use on average $n$ evaluations of $W$.

In the autonomous case, the assumptions from Theorem~\ref{thm:nonadapt} can be further weakened and we show that the statement holds if $\mu$ is only bounded on $[\xi - \delta, \xi + \delta]$ and the remaining assumptions from Theorem~\ref{thm:nonadapt} are fulfilled.

\begin{theorem}\label{thm:nonadapt:auton}
	Let $T \in (0, \infty)$, $\mu, \sigma\colon \R \rightarrow \R$ be measurable functions and let $X \colon [0, T] \times \Omega \rightarrow \R$ be an adapted process with continuous paths such that 
	\begin{equation*}
		X_t = X_0 + \int_0^t \mu(X_s) \, ds + \int_0^t \sigma(X_s) \, dW_s, \qquad t \in [0,T].
	\end{equation*}
	Assume that there exist $t_0 \in [0, T), T_0 \in (t_0, T], \delta \in (0, \infty)$ and $\xi \in \R$ such that it holds (reach) and
	\begin{itemize}[align=left,labelwidth=\widthof{(local reg)},leftmargin=\labelwidth+\labelsep]
		\item[(local reg)] $\mu$ is bounded on $[\xi - \delta, \xi + \delta]$, $\sigma$ is Lipschitz continuous on $[\xi - \delta, \xi + \delta]$,
		\item [(non-deg*)] $\inf_{x \in B_\delta(\xi)} |\sigma(x)| > 0$.
	\end{itemize}
	
	Then there exists a constant $c \in (0, \infty)$ such that for all $n \in \mathbb{N}$,
	\[
	\inf_{\substack{t_1, \dots, t_n \in [0,T] \\g \colon \mathbb{R}^n \rightarrow L^1([0,T]) \: measurable}} \EE\big[ \|X - g(W_{t_1}, \dots, W_{t_n})\|_{L^1([0,T])}\big] \ge \frac{c}{n^{1/2}}.
	\]
\end{theorem}

This picks up the spirit of~\cite{DGL22}, where it is shown in Theorem 1.2 that the Euler scheme for bounded drift coefficients and sufficiently regular diffusion coefficients reaches a rate of at least $1/2$ up to some small $\eps \in (0,1)$. Also for Theorem~\ref{thm:nonadapt:auton}, we show that the rate $1/2$ is optimal even for adaptive methods that use on average $n$ evaluations of $W$.

The paper is structured as follows. First, in Section~\ref{sect:intro:cplOfNoise}, we present the global coupling of noise technique from~\cite{MGY23} and the new local coupling of noise technique. After that, we introduce some notations in Section~\ref{sect:notation}. Then, in Section~\ref{sect:finalTimePointApprox}, we show how the local coupling of noise technique can be used in many situations to obtain lower bounds for the approximation error in the case of final time approximation and we prove Theorem~\ref{cor:pwLip}. In Section~\ref{sect:globalApprox}, we first introduce the class of adaptive methods and then we show how the local coupling of noise technique can also be applied for the derivation of lower bounds for such methods, thus proving Theorem~\ref{thm:nonadapt} and Theorem~\ref{thm:nonadapt:auton}.

\section{Introduction of the coupling of noise techniques}\label{sect:intro:cplOfNoise}
Let $(\Omega,\Fc, (\Fc_t)_{t \ge 0},\PP)$ be a filtered probability space satisfying the usual conditions and let $W\colon [0,\infty) \times \Omega\to \R$ be a standard Brownian motion. In this section, we first introduce the idea of global couplings and then we discuss the new local couplings.

The global coupling of noise technique has already been used in ~\cite{ELL24, EMGY25Hoelder, EMGY24Sobolev, MGY23} to derive lower bounds for approximation errors. There, one considers approximation errors for stochastic processes $X\colon [0,T] \times \Omega \rightarrow \R$, which are functionals of the Brownian motion $W$, i.e. there are a $\Fc_0$-measurable random variable $\eta_0$ and a Borel-measurable map $F \colon \mathbb{R} \times C([0,T]) \rightarrow C([0,T])$ such that
\begin{equation}\label{eqn_intro_2}
X = F(\eta_0, W).
\end{equation}

For a discretization $0 = t_0 < t_1 < \dots < t_n = T$, where $n \in \N$, one then considers the piecewise linear interpolation $\overline{W}$ of $W$ which is for $t \in [t_{i-1}, t_i]$, where $i \in \{1, \dots, n\}$, given by
\begin{equation}\label{eqn_intro_3}
\overline{W}_t = \frac{t - t_{i-1}}{t_i - t_{i-1}} W_{t_{i}} + \frac{t_i - t}{t_i - t_{i-1}} W_{t_{i-1}}
\end{equation}
and one sets
\[
B = W - \overline{W}.
\]
Then $(B_t)_{t\in [t_0, t_1]}, \dots, (B_t)_{t \in [t_{n-1}, t_n]}$ are independent Brownian bridges, which are independent of $\overline{W}$. Therewith, a new stochastic process $\tB$ is chosen such that it holds $\PP^{(\eta_0, \overline{W} , B)} = \PP^{(\eta_0, \overline{W}, \tB)}$.
This is used to define a new Brownian motion
\[
\tW = \overline{W} + \tB
\] 
and the global coupling
\[
\tX = F(\eta_0, \tW).
\]

First, we turn to the final time approximation. Due to $\PP^{(\eta_0, \overline{W} , B)} = \PP^{(\eta_0, \overline{W}, \tB)}$, an application of the triangle inequality shows that it holds for all measurable $g\colon \R^n \rightarrow \R$ and all $p \in [1, \infty)$
\[
	\big(\EE \bigl[|X_1 - \tX_1|^p \bigr] \big)^{1/p} \le 2\big( \EE \bigl[|X_1 - g(W_{t_1}, \dots, W_{t_n})|^p \bigr] \big)^{1 \slash p}.
\]
Subsequently, suitable bounds for $\big(\EE \bigl[|X_1 - \tX_1|^p \bigr] \big)^{1/p}$ are derived.  In ~\cite{ELL24, EMGY25Hoelder, EMGY24Sobolev, MGY23} problem specific estimates are needed for these bounds and in particular suitable upper bounds for the approximation error are required there, which seems unintuitive.

This problem no longer occurs with the local coupling of noise technique. The idea
is to consider not only couplings on the interval $[0,T]$, but also local couplings on $[t_{i-1}, t_i]$ for $i \in \{1, \dots, n\}$ which can be used to determine the asymptotic behavior of $\big(\EE \bigl[|X_1 - \tX_1|^p \bigr] \big)^{1/p}$. To be more precise, we assume that $(X_t)_{t \in [t_{i-1}, t_i]} = F_i(\eta_{i-1}, (W_{t} - W_{t_{i-1}})_{t \in [t_{i-1}, t_i]})$, where $\eta_{i-1}$  is a $\sigma(\{(W_t,\tW_t) \colon t\in[0, t_{i-1}]\})$-measurable random variable and $F_i\colon \R \times C([t_{i-1}, t_i]) \rightarrow C([t_{i-1}, t_i])$ is a measurable function, and we set
\[
\tX^{(i)} = F_i(\eta_{i-1}, (\tW_{t} - \tW_{t_{i-1}})_{t \in [t_{i-1}, t_i]}).
\]
The goal is then to show that a constant $c \in (0, \infty)$ exists such that
\begin{equation}\label{eqn_intro_1}
c \sum_{i=1}^n \EE \bigl[ |X_{t_i} - \tX^{(i)}_{t_i}|^p \bigr] \le \EE \bigl[|X_1 - \tX_1|^p\bigr].
\end{equation}
This is possible, for example, if $p=2$ and
\begin{itemize}[align=left,labelwidth=\widthof{(tansform))},leftmargin=\labelwidth+\labelsep] 
	
	\item [(transform)] there exists a bi-Lipschitz continuous transformation $G\colon \R \rightarrow \R$ of the SDE~\eqref{sde0} with an absolutely continuous derivative $G'$ such that the transformed coefficients $\tmu = \bigl(G'\mu + \frac{1}{2}D^2G \cdot \sigma^2 \bigr) \circ G^{-1}$ and $\tsigma = \bigl( G' \sigma\bigr) \circ G^{-1}$ are Lipschitz continuous where $D^2G$ is a weak derivative of $G'$.
\end{itemize}
Note that (transform) holds under the assumptions of Theorem~\ref{cor:pwLip}, cf. the proofs of Lemma 1 and Lemma 2 in~\cite{MGY19b}, and if it holds $\sup_{t \in \R} |\int_0^t \mu(z) \, dz| < \infty$, $\mu$ is bounded and $\sigma = 1$, see~\cite{EMGY25Hoelder}.

If (transform) is satisfied, $X$ is the solution of the SDE~\eqref{sde0} and one may choose $\tX^{(i)}$ as the solution of the SDE~\eqref{sde0} on the interval $[t_{i-1}, t_i]$ with initial value $X_{t_{i-1}}$ und driving Brownian motion $(\tW_{t} - \tW_{t_{i-1}})_{t \in [t_{i-1}, t_i]}$. Then the constant $c$ in \eqref{eqn_intro_1} exists, as we show in Corollary~\ref{cor:localCoupling}.

We proceed in a similar way with the global time approximation where we choose $\tB = -B$. For simplification, we assume
that $\mu, \sigma$ are Lipschitz continuous and that $\inf_{x \in \R} |\sigma(x)| > 0$ holds. We consider for $t \in (t_{i-1}, t_i]$ the Euler-type step
\[
\overline{X}_t = X_{t_{i-1}} + \sigma(X_{t_{i-1}})(W_t - W_{t_{i-1}})
\]
and the local coupling 
\[
\widetilde{\overline{X}}_t = X_{t_{i-1}} + \sigma(X_{t_{i-1}})(\tW_t - \tW_{t_{i-1}}).
\]
Then there is a constant $c_1 \in (0, \infty)$ such that it holds for all measurable $g\colon \R^n \rightarrow L^1([0,T])$ 
\begin{equation*}
	\begin{aligned}
		&\EE \big[ \|X - g(W_{t_1}, \dots, W_{t_n}) \|_{L^1([0,T])}\big] \\
		& \qquad \qquad = \int_0^T \EE \big[ |X_s - g(W_{t_1}, \dots, W_{t_n})(s)|\big] \, ds \\
		& \qquad \qquad \geq \sum_{i=1}^n \int_{t_{i-1}}^{t_i}\EE \big[ |\overline{X}_s - g(W_{t_1}, \dots, W_{t_n})(s)|\big] \, ds - \frac{c_1}{n}.
	\end{aligned}
\end{equation*} 
Again, an application of the triangle inequality yields due to $\PP^{(\overline{W} , B)} = \PP^{(\overline{W}, \tB)}$ and the independence of $X_{t_{i-1}}, (W_t - W_{t_{i-1}})_{t \in [t_{i-1}, t_i]}$ for all $i \in \{1, \dots, n\}$
\begin{equation*}
	\begin{aligned}
		&\EE \big[ \|X - g(W_{t_1}, \dots, W_{t_n}) \|_{L^1([0,T])}\big] \\		
		& \qquad \qquad \geq \frac{1}{2} \sum_{i=1}^n \int_{t_{i-1}}^{t_i}\EE \big[ |\overline{X}_s - \widetilde{\overline{X}}_s |\big] \, ds - \frac{c_1}{n}\\
		& \qquad \qquad = \frac{1}{2} \sum_{i=1}^n \int_{t_{i-1}}^{t_i}\EE \big[ |\sigma(X_{t_{i-1}})| \cdot |B_s - \tB_s|\big] \, ds - \frac{c_1}{n},
	\end{aligned}
\end{equation*} 
which gives the appropriate bound because of $\inf_{x \in \R} |\sigma(x)| > 0$ and $\tB = -B$.

\section{Notation}\label{sect:notation}
For $\xi \in \R$ and $\delta \in (0, \infty)$ we write $B_\delta(\xi)$ for the open ball in $\R$ around $\xi$ with radius $\delta$ and $\overline{B_{\delta}(\xi)}$ for its closure. Moreover, for $T \in (0, \infty)$ and $p \in [1, \infty)$ we use $L^p([0,T])$ for the space of measurable functions $f\colon [0,T] \rightarrow \R$ which satisfy $\| f\|_{L^p([0,T])} := \big(\int_0^T |f(x)|^p \, dx\big)^{1/p} < \infty$. The space of continuous functions $f\colon [0,T] \rightarrow \R$ is denoted by $C([0,T])$. Furthermore, the sign function  is given by
\[
\sgn\colon \R \rightarrow \{-1,0,1\}, \qquad x \mapsto 
\begin{cases}
	-1, & \text{if $x < 0$,} \\
	0, & \text{if $x = 0$,} \\
	1, & \text{if $x > 0$.}
\end{cases}
\]

For a probability space $(\Omega, \mathcal{F}, \PP)$ and $p \in [1, \infty)$ we denote by $L^p(\Omega, \mathcal{F}, \PP)$ the space of random variables $\eta\colon (\Omega, \mathcal{F}, \PP) \rightarrow \R$ that satisfy $\EE[|\eta|^p] < \infty$.

\section{Final time approximation}\label{sect:finalTimePointApprox}
	
	In this section, we first introduce the global coupling of noise technique from \cite{MGY23}. We then present the
	local coupling of noise technique and we show that the asymptotic behavior of the global coupling is determined by local couplings. Finally, this is used to prove Theorem~\ref{cor:pwLip}.

\subsection{Global coupling of noise}\label{subsect:globalCoupling}
	As seen in~\eqref{eqn_intro_2}, the global coupling of noise technique requires that the process, which is approximated, can be written as a functional of the Brownian motion $W$. If (transform) is satisfied, this is the case for a solution of the SDE~\eqref{sde0}, which is shown in the following lemma.
	\begin{lemma} \label{lemma:existenceOfFunctionF}
		Let $\mu, \sigma\colon \R \rightarrow \R$ be measurable functions satisfying (transform) with transformation $G\colon \R \rightarrow \R$. Then for every $T \in (0, \infty)$ there exists a Borel-measurable function
		\begin{flalign*}
			F \colon \mathbb{R} \times C([0,T]) \rightarrow C([0,T])
		\end{flalign*}
		such that for every complete probability space $(\Omega, \mathcal{F}, \mathbb{P})$, every Brownian motion $W \colon [0,T] \times \Omega \rightarrow \mathbb{R}$ and every random variable $\eta \colon \Omega \rightarrow \mathbb{R}$ such that $W, \eta$ are independent it holds:
		\begin{itemize}
			\item[(i)] if $X \colon [0,T] \times \Omega \rightarrow \mathbb{R}$ is a strong solution of the SDE \eqref{sde0} on the time interval $[0,T]$ with driving Brownian motion $W$ and initial value $\eta$, then $\mathbb{P}$-almost surely it holds $X = F(\eta, W)$,
			
			\item[(ii)] $F(\eta, W)$ is a strong solution of the SDE \eqref{sde0} on the time interval $[0,T]$ with driving Brownian motion $W$ and initial value $\eta$.
		\end{itemize}
	\end{lemma}

	\begin{proof}
		Note that $(G^{-1})' = \frac{1}{G' \circ G^{-1}}$ is absolutely continuous since $G^{-1}$ is Lipschitz continuous, $G'$ is absolutely continuous and bounded away from zero. Moreover, $D^2 G^{-1} = - \frac{D^2 G}{(G')^3} \circ G^{-1}$ is a weak derivative of $(G^{-1})'$ and transforming $\tmu, \tsigma$ with $G^{-1}$ yields the transformed coefficients 
		\[
		\bigl((G^{-1})'\tmu + \frac{1}{2}D^2 G^{-1} \cdot (\tsigma)^2 \bigr) \circ G = \mu \qquad \text{and} \qquad \bigl( (G^{-1})' \tsigma\bigr) \circ G = \sigma.
		\]
		Therewith, the claim follows with similar arguments as in the proof of Lemma 9 in~\cite{MGY23}.
	\end{proof}

	We now proceed similarly to Section  2.2 in~\cite{MGY23}. For the proof of Theorem~\ref{cor:pwLip} it suffices to consider discretizations $0 = t_0 < t_1 < \dots < t_{n-1} < t_n = 1$, where $n \in \N$, which satisfy
	\begin{equation}\label{maxDistanceOfTi}
		t_i - t_{i-1} \le \frac{2}{n}, \qquad \qquad i \in \{1, \dots, n\}.
	\end{equation}
	With the piecewise linear interpolation  $\overline{W}$ of $W$ from~\eqref{eqn_intro_3} we define
	\[
	B = W - \overline{W}.
	\]
	Then $(B_t)_{t\in [t_0, t_1]}, \dots, (B_t)_{t \in [t_{n-1}, t_n]}$ are Brownian bridges and 
	\[
	(B_t)_{t\in [t_0, t_1]}, \dots, (B_t)_{t \in [t_{n-1}, t_n]}, \overline{W}
	\]
	are independent. We choose new Brownian bridges $(\tB_t)_{t\in [t_0, t_1]}, \dots, (\tB_t)_{t \in [t_{n-1}, t_n]}$ such that
	\[
	(\tB_t)_{t\in [t_0, t_1]}, \dots, (\tB_t)_{t \in [t_{n-1}, t_n]}, W
	\]
	are independent and we set $\tB = (\tB_t)_{t \in [0,1]}$ as well as
	\[
	\tW = \overline{W} + \tB.
	\]
	Then $\tW$ is a Brownian motion and with Lemma~\ref{lemma:existenceOfFunctionF} we choose a solution $\tX$ of the SDE~\eqref{sde0} with initial value $x_0$ and driving Brownian motion $\tW$, assuming that $\mu, \sigma$ satisfy (transform). Using Lemma~\ref{lemma:existenceOfFunctionF} one may show similar to Lemma~11 in~\cite{MGY23} that the approximation error of a method based on the evaluations $W_{t_1}, \dots, W_{t_n}$ has the distance of the global couplings as a lower bound up to some constant. The formal statement can be seen in the next lemma.
	
	\begin{lemma}\label{lemma:lowerBound:infBoundedByTildeDistance}
		Let $\mu, \sigma \colon \mathbb{R} \rightarrow \mathbb{R}$ be measurable functions such that (transform) holds. Let $x_0 \in \mathbb{R}$ and $X,\widetilde{X} \colon [0,1] \times \Omega \rightarrow \mathbb{R}$ be strong solutions of the SDE \eqref{sde0} on the time interval $[0,1]$ with initial value $x_0$ and driving Brownian motion $W$ and $\widetilde{W}$, respectively. Then for every measurable function $g \colon \mathbb{R}^n \rightarrow \mathbb{R}$ and for every $p \in [1, \infty)$ it holds
		\begin{flalign*}
			\big( \EE \bigl[|X_1 - g(W_{t_1}, \dots, W_{t_n})|^p \bigr] \big)^{1 \slash p} \geq \frac{1}{2} \big(\EE \bigl[|X_1 - \tX_1|^p \bigr] \big)^{1/p}.
		\end{flalign*}
	\end{lemma}

In consideration of Lemma~\ref{lemma:lowerBound:infBoundedByTildeDistance}, the goal is to find suitable lower bounds for $(\EE[|X_1 - \tX_1|^p])^{1/p}$.

\subsection{Local coupling of noise}\label{subsect:localCoupling}

	In this section we assume that (transform) holds for measurable functions $\mu, \sigma \colon \R \rightarrow \R$. Then, because of the bi-Lipschitz continuity of $G$, there exists a constant $c \in (0, \infty)$ such that for all $p \in [1, \infty)$ it holds
	\[
	\big(\EE \bigl[|X_1 - \tX_1|^p \bigr] \big)^{1/p} \ge c \big(\EE \bigl[|G(X_1) - G(\tX_1)|^p \bigr] \big)^{1/p}.
	\]
	Now, using the It\^{o} formula, see e.g.~\cite[Problem 3.7.3]{ks91}, one sees that the processes 
	\[
	Y = (G(X_t))_{t \in [0,1]} \qquad \text{and} \qquad \tY = (G(\tX_t))_{t \in [0,1]}
	\]
	are solutions of the SDE
	\begin{flalign}
		dY_t = \tmu(Y_t) \, dt + \tsigma(Y_t) \, dW_t \label{sde1}
	\end{flalign}
	with initial value $G(x_0) \in \R$ and driving Brownian motion $W$ and $\tW$, respectively.
	
	Fix $i \in \{1, \dots, n\}$, set 
	\begin{equation*}
		\begin{aligned}
			V^{(i-1)} &:= (V^{(i-1)}_t)_{t\in [0, t_i - t_{i-1}]} := (W_{t + t_{i-1}} - W_{t_{i-1}})_{t \in [0, t_i - t_{i-1}]} , \\
			\tV^{(i-1)} &:= (\tV^{(i-1)}_t)_{t\in [0, t_i - t_{i-1}]} := (\tW_{t + t_{i-1}} - \tW_{t_{i-1}})_{t \in [0, t_i - t_{i-1}]}
		\end{aligned}
	\end{equation*}
	and let $Y^{(\eta, V^{(i-1)})}, \tY^{(\eta, \tV^{(i-1)})}$ denote strong solutions of the SDE~\eqref{sde1} on the time interval $[0, t_i-\nolinebreak t_{i-1}]$ with initial value $\eta$ and driving Brownian motion $V^{(i-1)}$ and $\tV^{(i-1)}$, respectively, where $\eta \in L^2(\Omega, \mathcal{F}, \PP)$ is independent of
	\[
	\Fc^{V^{(i-1)}, \tV^{(i-1)}} = \sigma(\{(V^{(i-1)}_t, \tV^{(i-1)}_t)\colon t\in [0, t_i - t_{i-1}]\}).
	\]
	
	Note that similar to Lemma 13 in~\cite{MGY23} 
	\begin{equation}\label{qx2}
		\forall i\in\{1,\dots,n\}\colon \, (X_{t_{i}}, \tX_{t_{i}}) \text{ and } \sigma\bigl(\{(W_t-W_{t_i},\tW_t-\tW_{t_i})\colon t\in[t_i,1]\}\bigr)\text{ are independent}.
	\end{equation}
	
	Thus, $\tY^{(Y_{t_{i-1}}, \tV^{(i-1)})}$ is a local coupling and with the following proposition we show later that the distance of $Y_1$ and the global coupling $\tY_1$ can be determined by the distances of $Y_{t_i}$ and the local coupling $\tY_{t_i - t_{i-1}}^{(Y_{t_{i-1}}, \tV^{(i-1)})}$.

	\begin{prop} \label{prop:localCoupling}
		There exist constants $c_1, c_2, c_3, c_4 \in (0, \infty)$, which are independent of $n$ and $i$, such that
		\begin{flalign*}
			\EE \bigl[ |Y_{t_i} - \tY_{t_i}|^2 \bigr] \geq (1 - \frac{c_1}{n}) \EE\bigl[ |Y_{t_{i-1}} - \tY_{t_{i-1}}|^2\bigr] + c_2 \EE \bigl[|Y_{t_i} - \tY_{t_i - t_{i-1}}^{(Y_{t_{i-1}}, \tV^{(i-1)})}|^2\bigr]
		\end{flalign*}
		and 
		\begin{flalign*}
			\EE \bigl[ |Y_{t_i} - \tY_{t_i}|^2 \bigr] \leq (1 + \frac{c_3}{n}) \EE \bigl[ |Y_{t_{i-1}} - \tY_{t_{i-1}}|^2 \bigr] + c_4 \EE \bigl[ |Y_{t_i} - \tY_{t_i - t_{i-1}}^{(Y_{t_{i-1}}, \tV^{(i-1)})}|^2 \bigr].
		\end{flalign*}
	\end{prop}
	
	\begin{proof}
		We use ideas from the proof of Lemma 11 in~\cite{ELL24}.
		
		Throughout this proof let $c_1,c_2,\dots\in (0,\infty)$ denote positive constants,
		which neither depend on $n$ nor on $i$. It holds 
		\begin{equation*}
			\begin{aligned}
				\EE\bigl[ |Y_{t_i} - \widetilde Y_{t_i}|^2\bigr] = \EE\bigl[ |Y_{t_{i-1}} - \widetilde Y_{t_{i-1}}|^2\bigr] + 2 m_i  + d_i, 
			\end{aligned}
		\end{equation*}
		where 
		\[
		m_i := \EE\bigl[ (Y_{t_{i-1}} - \widetilde Y_{t_{i-1}})((Y_{t_i} - Y_{t_{i-1}}) - (\widetilde Y_{t_i} - \widetilde Y_{t_{i-1}})) \bigr]
		\]
		and 
		\[
		d_i := \EE\bigl[ |(Y_{t_i} - Y_{t_{i-1}}) - (\widetilde Y_{t_i} - \widetilde Y_{t_{i-1}})|^2\bigr].
		\]
		Next, we show that 
		\begin{equation}\label{lemf3_1}
			|m_i| \leq \frac{c_1}{n} \EE\bigl[ |Y_{t_{i-1}} - \widetilde Y_{t_{i-1}}|^2\bigr]
		\end{equation}
		and
		\begin{equation}\label{lemf3_2}
			\begin{aligned}
				& d_i \ge c_2 \EE \bigl[ |Y_{t_i} - \tY_{t_i - t_{i-1}}^{(Y_{t_{i-1}}, \tV^{(i-1)})}|^2 \bigr] - \frac{c_3}{n} \EE\bigl[|Y_{t_{i-1}} - \widetilde Y_{t_{i-1}}|^2\bigr]
			\end{aligned}
		\end{equation}
		as well as 
		\begin{equation}\label{lemf3_3}
			\begin{aligned}
				& d_i \le c_4\EE \bigl[ |Y_{t_i} - \tY_{t_i - t_{i-1}}^{(Y_{t_{i-1}}, \tV^{(i-1)})}|^2 \bigr] + \frac{c_5}{n} \EE\bigl[|Y_{t_{i-1}} - \widetilde Y_{t_{i-1}}|^2\bigr]
			\end{aligned}
		\end{equation}
		which will yield the claim.
		
		Before estimating above terms, let us mention some properties of SDEs with Lipschitz continuous coefficients. Since $\tmu, \tsigma$ are Lipschitz continuous, there exists with Lemma~\ref{lemma:existenceOfFunctionF} a measurable function  
		\[
		F \colon \R \times C([0, t_i - t_{i-1}]) \rightarrow C([0, t_i - t_{i-1}])
		\]
		such that $F(\eta, \hat V)$ is a strong solution of the SDE~\eqref{sde1} with driving Brownian motion $\hat V \in \{V^{(i-1)}, \tV^{(i-1)}\}$ and initial value $\eta$ where $\eta \in L^2(\Omega, \mathcal{F}, \PP)$ is independent of $\Fc^{V^{(i-1)}, \tV^{(i-1)}}$.
		
		We will use a classical stability result for SDEs which states that there exists a constant $d \in (0, \infty)$, which is independent of $n$ and $i$, such that for all $\eta_1, \eta_2 \in L^2(\Omega, \mathcal{F}, \PP)$, which are independent of $\Fc^{V^{(i-1)}, \tV^{(i-1)}}$, it holds
		\begin{equation}\label{lemf3_4}
			\begin{aligned}
				&\sup_{s \in [0, t_i - t_{i-1}]} \EE \bigl[|Y^{(\eta_1, V^{(i-1)})}_s - Y^{(\eta_2, V^{(i-1)})}_s|^2 \bigr]\le d \EE \bigl[|\eta_1 - \eta_2|^2 \bigr], \\
				& \sup_{s \in [0, t_i - t_{i-1}]} \EE \bigl[|\tY^{(\eta_1, \tV^{(i-1)})}_s - \tY^{(\eta_2, \tV^{(i-1)})}_s|^2 \bigr] \le d \EE \bigl[|\eta_1 - \eta_2|^2 \bigr],
			\end{aligned}	
		\end{equation} 
		 see e.g. the proof of Theorem 9.2.4 in~\cite{RevuzYor1999}.
		
		Let us start with the proof of \eqref{lemf3_1}. Since it holds $Y_{t_i} = F(Y_{t_{i-1}}, V^{(i-1)})(t_i - t_{i-1}), \tY_{t_i} = F(\tY_{t_{i-1}}, \tV^{(i-1)})(t_i - t_{i-1})$ by the choice of $F$, we obtain using \eqref{maxDistanceOfTi}, \eqref{qx2} and \eqref{lemf3_4} for $\PP^{(Y_{t_{i-1}}, \tY_{t_{i-1}})}$-almost all $(y, \tilde y) \in \R \times \R$
		\begin{equation*}
			\begin{aligned}
				&\big|\EE\bigl[ (Y_{t_{i-1}} - \widetilde Y_{t_{i-1}})((Y_{t_i} - Y_{t_{i-1}}) - (\widetilde Y_{t_i} - \widetilde Y_{t_{i-1}})) \vert (Y_{t_{i-1}} , \tY_{t_{i-1}}) = (y, \tilde y)\bigr]\big|\\
				&\qquad \qquad = |y - \tilde y| \cdot \big| \EE\bigl[ (F(y, V^{(i-1)})(t_i - t_{i-1}) - y) - (F(\tilde y, \tV^{(i-1)})(t_i - t_{i-1}) - \tilde y) \bigr]\big|\\
				&\qquad \qquad = |y - \tilde y| \cdot \big|\EE\bigl[ (F(y, V^{(i-1)})(t_i - t_{i-1}) - y) - (F(\tilde y, V^{(i-1)})(t_i - t_{i-1}) - \tilde y) \bigr]\big|\\
				&\qquad \qquad = |y - \tilde y| \cdot \big|\EE\bigl[ (Y_{t_i - t_{i-1}}^{(y, V^{(i-1)})} - y) - (Y_{t_i - t_{i-1}}^{(\tilde y, V^{(i-1)})} - \tilde y) \bigr]\big|  \\
				&\qquad \qquad = |y - \tilde y| \cdot \big|\EE \bigl[\int_0^{t_i - t_{i-1}} \tmu(Y_s^{(y, V^{(i-1)})}) - \tmu(Y_s^{(\tilde y, V^{(i-1)})}) \, ds \bigl]\big| \\
				&\qquad \qquad \leq \frac{c_6}{n} \cdot |y - \tilde y|^2.
			\end{aligned}
		\end{equation*}
		This shows \eqref{lemf3_1} and we continue with the derivation of \eqref{lemf3_2} and \eqref{lemf3_3}. Note that it holds
		\begin{equation}\label{lemf3_5}
			d_i \ge \frac{1}{2} \EE\bigl[ |Y_{t_i} - \tY^{(Y_{t_{i-1}}, \tV^{(i-1)})}_{t_i - t_{i-1}}|^2 \bigr] -  \EE\bigl[ |( \tY^{( Y_{t_{i-1}}, \tV^{(i-1)})}_{t_i - t_{i-1}} - Y_{t_{i-1}}) - ( \tY_{t_i} -  \tY_{t_{i-1}})|^2\bigr]
		\end{equation}
		as well as 
		\begin{equation}\label{lemf3_6}
			d_i \le 2 \EE\bigl[ |Y_{t_i} - \tY^{(Y_{t_{i-1}}, \tV^{(i-1)})}_{t_i - t_{i-1}}|^2 \bigr] + 2 \EE\bigl[ |( \tY^{(Y_{t_{i-1}}, \tV^{(i-1)})}_{t_i - t_{i-1}} - Y_{t_{i-1}}) - ( \tY_{t_i} -  \tY_{t_{i-1}})|^2\bigr].
		\end{equation}
		By the choice of $F$, an application of \eqref{maxDistanceOfTi}, \eqref{qx2} and \eqref{lemf3_4} shows that it holds for $\PP^{(Y_{t_{i-1}}, \tY_{t_{i-1}})}$-almost all $(y, \tilde y) \in \R \times \R$
		\begin{equation*}
			\begin{aligned}
				&\EE\bigl[ |( \tY^{(Y_{t_{i-1}}, \tV^{(i-1)})}_{t_i - t_{i-1}} - Y_{t_{i-1}}) - ( \tY_{t_i} -  \tY_{t_{i-1}})|^2 \vert (Y_{t_{i-1}} , \tY_{t_{i-1}}) = (y, \tilde y) \bigr] \\
				& \qquad  = \EE\Bigl[ \Bigl| \int_0^{t_i - t_{i-1}} \tmu(\tY^{(y, \tV^{(i-1)})}_s) - \tmu(\tY^{(\tilde y, \tV^{(i-1)})}_s)\, ds \\
				& \qquad \qquad \qquad + \int_0^{t_i - t_{i-1}} \tsigma(\tY^{(y, \tV^{(i-1)})}_s) - \tsigma(\tY^{(\tilde y, \tV^{(i-1)})}_s)\, d\tV^{(i-1)}_s \Bigr|^2\Bigr] \\
				&\qquad \le c_7\EE\Bigl[  \int_0^{t_i - t_{i-1}} | \tY^{(y, \tV^{(i-1)})}_s - \tY^{(\tilde y, \tV^{(i-1)})}_s |^2\, ds\Bigr] \\
				& \qquad \le \frac{c_8}{n} \cdot  |y - \tilde y|^2,
			\end{aligned}
		\end{equation*}
		and thus we obtain
		\begin{equation*}
			\EE\bigl[ |( \tY^{(Y_{t_{i-1}}, \tV^{(i-1)})}_{t_i - t_{i-1}} - Y_{t_{i-1}}) - ( \tY_{t_i} -  \tY_{t_{i-1}})|^2\bigr] \le \frac{c_8}{n} \EE \bigl[ |Y_{t_{i-1}} - \tY_{t_{i-1}}|^2 \bigr].
		\end{equation*}
		Together with \eqref{lemf3_5} and \eqref{lemf3_6} this yields \eqref{lemf3_2} and \eqref{lemf3_3} which finishes the proof.
	\end{proof}
	
Proposition~\ref{prop:localCoupling} can be used now to determine the asymptotic behavior of $\EE \bigl[|Y_1 - \tY_1|^2 \bigr]$ using local couplings.

\begin{cor}\label{cor:localCoupling}
	There exist constants $c_1, c_2 \in (0, \infty)$ and $N \in \N$ such that 
	\[
	c_1 \sum_{i=1}^n \EE \bigl[ |Y_{t_i} - \tY_{t_i - t_{i-1}}^{(Y_{t_{i-1}}, \tV^{(i-1)})}|^2 \bigr] \le \EE \bigl[|Y_1 - \tY_1|^2 \bigr] \le c_2 \sum_{i=1}^n \EE \bigl[ |Y_{t_i} - \tY_{t_i - t_{i-1}}^{(Y_{t_{i-1}}, \tV^{(i-1)})}|^2 \bigr],
	\]
	if $n \ge N$.
\end{cor}

\begin{proof}
	In the following $c_1,c_2,\dots\in (0,\infty)$ denote positive constants, which do not depend on $n$. With $c_1, c_2, c_3, c_4$ as in Proposition~\ref{prop:localCoupling} it follows by induction, if $n > c_1$,
	\begin{equation*}
		\begin{aligned}
			\EE \bigl[|Y_1 - \tY_1|^2 \bigr] &\ge c_2 \sum_{i=1}^n (1-\frac{c_1}{n})^{n-i} \EE \bigl[ |Y_{t_i} - \tY_{t_i - t_{i-1}}^{(Y_{t_{i-1}}, \tV^{(i-1)})}|^2 \bigr] \\
			&\ge c_2 (1 - \frac{c_1}{n})^n \sum_{i=1}^n \EE \bigl[ |Y_{t_i} - \tY_{t_i - t_{i-1}}^{(Y_{t_{i-1}}, \tV^{(i-1)})}|^2 \bigr]
		\end{aligned}
	\end{equation*}
	and
	\begin{equation*}
		\begin{aligned}
			\EE \bigl[|Y_1 - \tY_1|^2 \bigr] &\le c_4 \sum_{i=1}^n (1+\frac{c_3}{n})^{n-i} \EE \bigl[ |Y_{t_i} - \tY_{t_i - t_{i-1}}^{(Y_{t_{i-1}}, \tV^{(i-1)})}|^2\bigr] \\
			&\le c_4 (1 + \frac{c_3}{n})^n \sum_{i=1}^n \EE \bigl[ |Y_{t_i} - \tY_{t_i - t_{i-1}}^{(Y_{t_{i-1}}, \tV^{(i-1)})}|^2 \bigr].
		\end{aligned}
	\end{equation*}
	Because of $\lim_{n \rightarrow \infty} (1 - \frac{c_1}{n})^n = e^{-c_1}$ and $\lim_{n \rightarrow \infty} (1 + \frac{c_3}{n})^n = e^{c_3}$, the claim follows.
\end{proof}

\subsection{Proof of Theorem~\ref{cor:pwLip}}
We now turn to the proof of Theorem~\ref{cor:pwLip}. Therefore, we show a more general statement, where (A1)-(A3) must hold only locally around $\xi_i$ and the coefficients fulfill (transform). Note that under the assumptions of Theorem~\ref{cor:pwLip}, (transform) is satisfied, cf. the proofs of Lemma 1 and Lemma 2 in~\cite{MGY19b}.

\begin{theorem}\label{thm:lowebound}
	Assume that it holds (transform) for measurable functions $\mu, \sigma \colon \R \rightarrow \R$ and	\begin{itemize}[align=left,labelwidth=\widthof{(jump)},leftmargin=\labelwidth+\labelsep]
		\item[(jump)] 
		there exist $\delta \in (0, \infty), \xi \in \R$ such that \begin{itemize}[align=left,labelwidth=\widthof{(jump3)},leftmargin=\labelwidth+\labelsep]
			\item [(jump1)] $\mu$ is Lipschitz continuous on $[\xi - \delta, \xi)$ and on $(\xi, \xi + \delta]$,
			\item [(jump2)] it holds $\inf_{x \in B_\delta(\xi)} |\sigma(x)| > 0$, $\sigma$ is Lipschitz continuous on $[\xi - \delta, \xi + \delta]$, its derivative $\sigma'$ exists on $B_\delta(\xi) \setminus \{\xi\}$ and $\sigma'$ is Lipschitz continuous on $(\xi - \delta, \xi)$ and on $(\xi, \xi + \delta)$, respectively,
			\item[(jump3)] it holds $\bigl(\frac{\mu}{\sigma} - \frac{\sigma'}{2}\bigr)(\xi+) \neq \bigl(\frac{\mu}{\sigma} - \frac{\sigma'}{2}\bigr)(\xi-)$.
		\end{itemize}  
	\end{itemize}
	
	Let $x_0 \in \mathbb{R}$ and let $X\colon [0,1] \times \Omega \rightarrow \mathbb{R}$ be a strong solution of the SDE~\eqref{sde0} on the time interval $[0,1]$ with initial value $x_0$ and driving Brownian motion $W$ such that 
	\begin{itemize}[align=left,labelwidth=\widthof{(reach jump)},leftmargin=\labelwidth+\labelsep]  
		\item[(reach jump)] there exists a $t^\ast \in (0,1]$ with $p_{X_{t^\ast}}(\xi) > 0$.
	\end{itemize}
	
	Then there exists a constant $c \in (0, \infty)$ such that for all $n \in \mathbb{N}$,
	\begin{flalign*}
		\inf_{\substack{t_1, \dots, t_n \in [0,1] \\g \colon \mathbb{R}^n \rightarrow \mathbb{R} \: measurable}} \big(\EE \bigl[|X_1 - g(W_{t_1}, \dots, W_{t_n})|^2 \bigr]\big)^{1/2} \geq \frac{c}{n^{3 \slash 4}}.
	\end{flalign*} 
\end{theorem}

Regarding (reach jump), we denote the local density of $X_t$ on $B_\delta(\xi)$ by $p_{X_t}$ which satisfies for all Borel-measurable sets $A \subset \R$
\[
\PP(X_t \in A \cap B_\delta(\xi)) = \int_{A \cap B_\delta(\xi)} p_{X_t}(x) \, dx.
\]
Moreover, we assume that the function
\[
(0,1] \times B_\delta(\xi) \rightarrow \R, \qquad (t,x) \mapsto p_{X_t}(x)
\]
is continuous, cf. Corollary 2 in~\cite{ELL25Density}.

In the next proposition we show Theorem~\ref{thm:lowebound} for the special case $\sigma\vert_{[\xi- \delta, \xi + \delta]} = 1$. The more general statement of Theorem~\ref{thm:lowebound} then follows with a Lamperti-type transformation.

\begin{prop}\label{prop:lowerbound:localAdditive}
	Let $\mu, \sigma\colon \R \rightarrow \R$  be measurable functions. Assume that (jump), (transform) and $\sigma\vert_{[\xi- \delta, \xi + \delta]} = 1$ hold. Let $x_0 \in \mathbb{R}$ and let $X\colon [0,1] \times \Omega \rightarrow \mathbb{R}$ be a strong solution of the SDE~\eqref{sde0} on the time interval $[0,1]$ with initial value $x_0$ and driving Brownian motion $W$ such that (reach jump) is satisfied. Then there exists a constant $c \in (0, \infty)$ such that for all $n \in \mathbb{N}$,
	\begin{flalign*}
		\inf_{\substack{t_1, \dots, t_n \in [0,1] \\g \colon \mathbb{R}^n \rightarrow \mathbb{R} \: measurable}} \big( \EE \bigl[|X_1 - g(W_{t_1}, \dots, W_{t_n})|^2 \bigr] \big)^{1/2}\geq \frac{c}{n^{3 \slash 4}}.
	\end{flalign*} 
\end{prop}

Below, we use the notations of Section~\ref{subsect:globalCoupling} and Section~\ref{subsect:localCoupling}. Since we want to apply Corollary~\ref{cor:localCoupling}, we need lower bounds for the distance between $Y_{t_i}$ and the local coupling $\tY_{t_i - t_{i-1}}^{(Y_{t_{i-1}}, \tV^{(i-1)})}$. A suitable bound for this is shown in the following lemma by localizing the problem.

\begin{lemma}\label{lem:lowerBound:coupled}
	Let the assumptions of Proposition~\ref{prop:lowerbound:localAdditive} hold. Then there exist a constant $c \in (0, \infty)$ and $N \in \N$ such that for all $i \in \{1, \dots, n\}$ it holds
	\begin{flalign*}
		\EE \bigl[|Y_{t_i} - \tY_{t_i - t_{i-1}}^{(Y_{t_{i-1}}, \tV^{(i-1)})}|^2 \bigr] \geq c (t_i - t_{i-1})^2 \cdot \mathbb{P}(X_{t_{i-1}} \in [\xi - \sqrt{t_i - t_{i-1}}, \xi + \sqrt{t_i - t_{i-1}}]),
	\end{flalign*}
	if $n \ge N$.
\end{lemma}

\begin{proof}
	Let $i \in \{1, \dots, n\}$. Throughout this proof let $c_1,c_2,\dots\in (0,\infty)$ denote positive constants, which neither depend on $n$ nor on $i$.
	
	The main idea of this proof is to use that the solution $X$ behaves locally as the solution of an SDE with piecewise Lipschitz continuous coefficients if the starting value of the SDE is close to the jump position $\xi$. The claim will then follow with already known results for the approximation of such regular SDEs.
	
	Note because of $\sigma\vert_{[\xi- \delta, \xi + \delta]} = 1$, (jump1) and (jump3) there exist $\gamma_1^\ast, \gamma_2^\ast \in \R$ with $\gamma_1^\ast \neq 0$ and a Lipschitz continuous function $\mu_{Lip}^\ast \colon \R \rightarrow \R$ such that
	\begin{flalign*}
		\mu(x) = \gamma_1^\ast 1_{[\xi, \infty)}(x) + \gamma_2^\ast 1_{\{\xi\}}(x) + \mu_{Lip}^\ast(x), \qquad \qquad  x \in [\xi - \delta, \xi + \delta],
	\end{flalign*}
	see Lemma 1 in~\cite{ELL24}. Set $\mu^\ast := \gamma_1^\ast 1_{[\xi, \infty)} + \gamma_2^\ast 1_{\{\xi\}} + \mu_{Lip}^\ast$.
	
	Let $x \in [\xi - \sqrt{t_i - t_{i-1}}, \xi + \sqrt{t_i - t_{i-1}}]$ and let $X^{\ast, x}, \tX^{\ast, x}$ denote solutions of the SDE with drift coefficient $\mu^\ast$, diffusion coefficient $\sigma^\ast = 1_\R$, initial value $x$ and driving Brownian motion $V^{(i-1)}$ and $\tV^{(i-1)}$, respectively. Analogously, let $X^{x}, \tX^{x}$ denote strong solutions of the SDE~\eqref{sde0} on the time interval $[0, t_i - t_{i-1}]$ with initial value $x$ and driving Brownian motion $V^{(i-1)}$ and $\tV^{(i-1)}$, respectively. Now we show that $X^x$ and $X^{\ast, x}$ as well as $\tX^x$ and $\tX^{\ast, x}$ coincide on some small time interval. Therefore, we define for $I := (\xi - (t_i - t_{i-1})^{1/4}, \xi + (t_i - t_{i-1})^{1/4})$ stopping times 
	\begin{equation*}
		\begin{aligned}
			\tau^x &:= \inf\{s \in [0, t_i - t_{i-1}] \colon X^x_s \notin I\} \wedge \inf\{s \in [0, t_i - t_{i-1}] \colon X^{\ast, x}_s \notin I\} \wedge (t_i - t_{i-1}), \\
			\ttau^x &:= \inf\{s \in [0, t_i - t_{i-1}] \colon \tX^x_s \notin I\} \wedge \inf\{s \in [0, t_i - t_{i-1}] \colon \tX^{\ast, x}_s \notin I\} \wedge (t_i - t_{i-1}).
		\end{aligned}
	\end{equation*} 
	Assume that $(2/n)^{1/4} < \delta$. With similar arguments as in the proofs of Lemma 1 and Lemma~2 in~\cite{MGY19b} there exists a bi-Lipschitz continuous function $G^\ast\colon \R \rightarrow \R$ with Lipschitz continuous derivative $(G^\ast)'$ such that the transformed coefficients $\widetilde{\mu^\ast}=  \bigl((G^\ast)' \mu^\ast  + \frac{1}{2} D^2G^\ast \cdot (\sigma^\ast)^2\bigr) \circ (G^\ast)^{-1}$ and $ \widetilde{\sigma^\ast}= \bigl((G^\ast)' \sigma^\ast \bigr) \circ (G^\ast)^{-1}$ are Lipschitz continuous where $D^2G^\ast$ is a weak derivative of $(G^\ast)'$. Due to \eqref{maxDistanceOfTi} and $(2/n)^{1/4} < \delta$ we have $I \subset B_\delta(\xi)$. Hence, Lemma~\ref{lem:localJump:stopped} yields that it holds $\PP$-almost surely for all $t \in [0,1]$
	\begin{equation}\label{eqn:lem1:1}
	X^x(t \wedge \tau^x) = X^{\ast, x}(t \wedge \tau^x) \qquad \text{and} \qquad \tX^x(t \wedge \ttau^x) = \tX^{\ast, x}(t \wedge \ttau^x).
	\end{equation}

	We are ready to prove the claim now. It holds due to \eqref{qx2}, (transform) and Lemma~\ref{lemma:existenceOfFunctionF}
	\begin{equation}\label{eqn:lem1:3}
		\begin{aligned}
			\EE \bigl[|Y_{t_i} - \tY_{t_i - t_{i-1}}^{(Y_{t_{i-1}}, \tV^{(i-1)})}|^2 \bigr] \geq \int_{\xi - \sqrt{t_i - t_{i-1}}}^{\xi + \sqrt{t_i - t_{i-1}}} 	\EE \bigl[|G(X^x_{t_i - t_{i-1}}) - G(\tX^x_{t_i - t_{i-1}})|^2 \bigr] \, \PP^{X_{t_{i-1}}}(dx).
		\end{aligned}
	\end{equation}	
	Using (transform), \eqref{eqn:lem1:1}, $\sigma^\ast = 1$ and the facts that $X^{\ast,x}_t$ and $\tX^{\ast, x}_t$ have densities for $t \in (0, t_i - t_{i-1}]$ as well as $V^{(i-1)}_{t_i - t_{i-1}} = W_{t_i} - W_{t_{i-1}} = \tV^{(i-1)}_{t_i - t_{i-1}}$, we obtain similarly to~\cite{ELL24}
	\begin{equation}\label{eqn:lem1:4}
		\begin{aligned}
			& \EE \bigl[|G(X^x_{t_i - t_{i-1}}) - G(\tX^x_{t_i - t_{i-1}})|^2 \bigr] \\
			&  \geq c_1 \EE \bigl[1_{\{\tau^x \wedge \ttau^x = t_i  - t_{i-1}\}}|X_{t_i - t_{i-1}}^{\ast, x} - \tX_{t_i - t_{i-1}}^{\ast, x}|^2 \bigr] \\
			& = c_1 \EE \big[1_{\{\tau^x \wedge \ttau^x = t_i  - t_{i-1}\}}|\gamma_1^\ast \int_0^{t_i - t_{i-1}} 1_{[\xi, \infty)}(X^{\ast, x}_s) - 1_{[\xi, \infty)}(\tX^{\ast, x}_s) \, ds \\
			& \qquad + \int_0^{t_i - t_{i-1}} \mu_{Lip}^\ast(X^{\ast, x}_s) - \mu_{Lip}^\ast(\tX^{\ast, x}_s) \, ds|^2 \big].
		\end{aligned}
	\end{equation}
	Now the part with the Lipschitz continuous function can be handled by
	\begin{equation}\label{eqn:lem1:5}
		\begin{aligned}
			&\EE \bigl[|\int_0^{t_i - t_{i-1}} \mu_{Lip}^\ast(X^{\ast, x}_s) - \mu_{Lip}^\ast(\tX^{\ast, x}_s) \, ds|^2 \bigr] \\
			& \qquad \qquad \leq c_2(t_i - t_{i-1}) \int_0^{t_i - t_{i-1}} \EE \bigl[|X_s^{\ast, x} - \tX_s^{\ast, x}|^2 \bigr] \, ds \leq c_3 (t_i - t_{i-1})^3.
		\end{aligned}
	\end{equation}
	Moreover, note that we have similar to the proof of Lemma 14 in~\cite{MGY23} for $p = 5$
	\begin{equation}\label{eqn:lem1:6}
		\begin{aligned}
			&\EE \bigl[|\int_0^{t_i - t_{i-1}} 1_{[\xi, \infty)}(X^{\ast, x}_s) - 1_{[\xi, \infty)}(x + V^{(i-1)}_s) \, ds|^2 \bigr] \\
			& \qquad \qquad  \leq c_4(t_i - t_{i-1}) \int_0^{t_i - t_{i-1}} \mathbb{P}(|x + V^{(i-1)}_s| \leq (t_i - t_{i-1})^{3/4}) \\
			& \qquad \qquad \qquad \qquad \qquad \qquad \qquad \qquad   + \mathbb{P}((t_i - t_{i-1})^{3/4} \leq |x + V^{(i-1)}_s - X^{\ast, x}_s |) \, ds \\
			&\qquad \qquad \leq c_5(t_i - t_{i-1}) \int_0^{t_i - t_{i-1}} \frac{(t_i - t_{i-1})^{3/4}}{\sqrt{s}} + (t_i - t_{i-1})^{-3p/4}\EE \bigl[|x + V^{(i-1)}_s - X^{\ast, x}_s |^p \bigr] \, ds \\
			&\qquad \qquad \leq c_6 (t_i - t_{i-1})((t_i - t_{i-1})^{5/4} + (t_i - t_{i-1})^{-3p/4 + p}) \leq c_{7} (t_i - t_{i-1})^{2 + 1/4}.
		\end{aligned} 
	\end{equation}
	Since for all $p \in [1, \infty)$ there exists a $C^{(p)} \in (0, \infty)$ such that for $Z^x \in \{X^x, \tX^x, X^{\ast, x}, \tX^{\ast, x}\}$ it holds
	\begin{equation*}
		\begin{aligned}
			&\PP(\inf\{s \in [0, t_i - t_{i-1}] \colon Z^x_s \notin I\} < t_i - t_{i-1}) \\
			&\qquad \qquad \leq \PP(\sup_{s \in [0, t_i - t_{i-1}]}|Z_s^x - \xi| \ge (t_i - t_{i-1})^{1/4}) \\
			&\qquad \qquad \leq (t_i - t_{i-1})^{-p/4} \cdot 2^p \cdot \bigl(\EE \bigl[\sup_{s \in [0, t_i - t_{i-1}]}|Z_s^x - x|^p \bigr] + |x - \xi|^p \bigr) \\
			& \qquad \qquad \leq C^{(p)}(t_i - t_{i-1})^{p/4}
		\end{aligned}
	\end{equation*}
	and thus
	\[
	\PP(\tau^x \wedge \ttau^x < t_i  - t_{i-1}) \le 4C^{(p)}(t_i - t_{i-1})^{p/4},
	\]
	it suffices in consideration of \eqref{maxDistanceOfTi}, \eqref{eqn:lem1:3}, \eqref{eqn:lem1:4}, \eqref{eqn:lem1:5}, \eqref{eqn:lem1:6} to show 
	\begin{equation}\label{eqn:lem1:2}
		\EE \bigl[|\int_0^{t_i - t_{i-1}} 1_{[\xi, \infty)}(x + V^{(i-1)}_s) - 1_{[\xi, \infty)}(x + \tV^{(i-1)}_s) \, ds|^2 \bigr] \geq c_{8}(t_i - t_{i-1})^2.
	\end{equation}
	Using
	\begin{equation*}
	\begin{aligned}
		&\EE \bigl[|\int_0^{t_i - t_{i-1}} 1_{[\xi, \infty)}(x + V^{(i-1)}_s) - 1_{[\xi, \infty)}(x + \tV^{(i-1)}_s) \, ds|^2 \bigr] \\
		&\qquad \qquad = \EE \big[|\int_0^{t_i - t_{i-1}} 1_{[\xi - x, \infty)}(\frac{s}{t_i - t_{i-1}} (W_{t_i} - W_{t_{i-1}}) + B_{t_{i-1} + s}) \\
		&\qquad \qquad \qquad \qquad - 1_{[\xi - x, \infty)}(\frac{s}{t_i - t_{i-1}} (W_{t_i} - W_{t_{i-1}}) + \tB_{t_{i-1} + s}) \, ds|^2\big],
	\end{aligned}
	\end{equation*}
	the bound in \eqref{eqn:lem1:2} follows similar to Lemma 3 in~\cite{MGY23}.
\end{proof}

Now we are able to prove Proposition~\ref{prop:lowerbound:localAdditive}.

\begin{proof}[Proof of Proposition~\ref{prop:lowerbound:localAdditive}]
	Let $c_1,c_2,\dots\in (0,\infty)$ denote positive constants, which do not depend on $n$. 
	
	It holds by Corollary~\ref{cor:localCoupling} and Lemma~\ref{lem:lowerBound:coupled} for all sufficiently large $n \in \N$
	\begin{equation}\label{eqn:thm1:1}
		\EE \bigl[|Y_1 - \tY_1|^2 \bigr] \ge c_1 \sum_{i=1}^n (t_i - t_{i-1})^2 \cdot \mathbb{P}(X_{t_{i-1}} \in [\xi - \sqrt{t_i - t_{i-1}}, \xi + \sqrt{t_i - t_{i-1}}]).
	\end{equation}
	Note that by (jump), (reach jump) and Corollary~2 in~\cite{ELL25Density} there exist $c^\ast, \delta^\ast \in (0, \infty)$ and $t_0^\ast, t_1^\ast \in (0,1)$ with $t_0^\ast < t_1^\ast$ such that 
	\begin{equation*}
		p_{X_t}(x) \ge c^\ast, \qquad t \in [t_0^\ast, t_1^\ast], \, x \in B_{\delta^\ast}(\xi).
	\end{equation*}
	Therefore it holds with \eqref{eqn:thm1:1} for all sufficiently large $n \in \N$
	\begin{equation}\label{eqn:thm1:2}
		\EE \bigl[|Y_1 - \tY_1|^2 \bigr] \ge c_1 \sum_{\substack{i \in \{1, \dots, n\} \\ t_0^\ast \le t_{i-1} \le t_1^\ast}} c^\ast (t_i - t_{i-1})^{5/2}.
	\end{equation}
	Now we have by \eqref{maxDistanceOfTi} and the H\"{o}lder inequality, similar to~\cite{MGY23},
	\[
	(t_1^\ast - t_0^\ast - \frac{2}{n}) \le \sum_{\substack{i \in \{1, \dots, n\} \\ t_0^\ast \le t_{i-1} \le t_1^\ast}} (t_i - t_{i-1}) \le n^{3/5} \cdot \Big( \sum_{\substack{i \in \{1, \dots, n\} \\ t_0^\ast \le t_{i-1} \le t_1^\ast}} (t_i - t_{i-1})^{5/2} \Big)^{2/5}
	\]
	and therefore we obtain with \eqref{eqn:thm1:2} for sufficiently large $n \in \N$
	\[
	\EE \bigl[|Y_1 - \tY_1|^2 \bigr] \ge c_2 n^{-3/2}.
	\]
	The claim now follows with Lemma~\ref{lemma:lowerBound:infBoundedByTildeDistance}, the bi-Lipschitz continuity of $G$ and the choice of $Y, \tY$.
\end{proof}

With Proposition~\ref{prop:lowerbound:localAdditive} and a suitable transformation, Theorem~\ref{thm:lowebound} can be shown now.

\begin{proof}[Proof of Theorem~\ref{thm:lowebound}]
	For the proof of the theorem, we transform the solution $X$ similar to~\cite{ELL25Density} with a local Lamperti-type transform to a solution of an SDE that satisfies the assumptions of Proposition~\ref{prop:lowerbound:localAdditive}. 
	
	The Lamperti-type transform is defined by
	\[
	H\colon \R \rightarrow \R, \qquad x \mapsto \int_0^x \frac{1}{\sigma^\ast(z)} \, dz,
	\]
	where $\sigma^\ast\colon \R \rightarrow \R$ is the constant continuation of $\sigma\vert_{[\xi-\delta, \xi+\delta]}$ given by
	\[
	\sigma^\ast = \sigma(\xi - \delta) 1_{(-\infty, \xi - \delta)} + \sigma 1_{[\xi-\delta, \xi + \delta]} + \sigma(\xi + \delta) 1_{(\xi + \delta, \infty)}.
	\]
	Since $\inf_{x \in B_\delta(\xi)} |\sigma(x)| > 0$ and  $\sigma$ is Lipschitz continuous on $[\xi - \delta, \xi + \delta]$, $H$ is bi-Lipschitz continuous and strictly monotonic. Moreover, by (jump2) $H' = \frac{1}{\sigma^\ast}$ is absolutely continuous and
	\begin{equation}\label{eqn:thm1:3}
		H''(x) = \frac{- (\sigma^\ast)'(x)}{(\sigma^\ast)^2(x)} = 1_{B_\delta(\xi)}(x) \cdot \frac{- \sigma'(x)}{\sigma^2(x)}, \qquad x \in \R \setminus \{\xi-\delta, \xi, \xi+\delta\}.
	\end{equation}
	So by a generalized It\^{o} formula, see e.g.~\cite[Problem 3.7.3]{ks91}, the transformed process $Z = H(X)$ is a strong solution of the SDE 
	\begin{equation*}
		\begin{aligned}
			dZ_t & = \mu^H(Z_t) \, dt +  \sigma^H(Z_t) \, dW_t, \quad t\in [0,1],\\
			Z_0 & = H(x_0),
		\end{aligned}
	\end{equation*}
	with $\mu^H = \bigl(H'\mu + \frac{1}{2}D^2H \cdot \sigma^2 \bigr) \circ H^{-1}$ and $\sigma^H = \bigl( H' \sigma\bigr) \circ H^{-1}$ where $D^2H = -1_{B_\delta(\xi) \setminus \{\xi\}} \frac{\sigma'}{\sigma^2}$ is a weak derivative of $H'$ due to~\eqref{eqn:thm1:3}. It remains to show that all assumptions which are needed for Proposition~\ref{prop:lowerbound:localAdditive} are satisfied by $Z$. By the strict monotonicity and continuity of $H$, there exists for $\xi^H = H(\xi)$ a $\delta^H \in (0, \infty)$ such that $[\xi^H - \delta^H, \xi^H + \delta^H] \subsetneq H([\xi-\delta, \xi+\delta])$ and it holds $\sigma^H(y)=\frac{\sigma}{\sigma^\ast} \circ H^{-1}(y)= 1$ for all $y \in [\xi^H - \delta^H, \xi^H+\delta^H]$.
	
	We continue with the verification of the assumption (jump).
	By (jump1), (jump2), \eqref{eqn:thm1:3} and the bi-Lipschitz continuity of $H$, $\mu^H$ is Lipschitz continuous on $[\xi^H - \delta^H, \xi^H)$ and on $(\xi^H, \xi^H + \delta^H]$ and hence $\mu^H$ satisfies (jump1). Since $\sigma^H\vert_{[\xi^H - \delta^H, \xi^H+\delta^H]} = 1$ also (jump2) holds for $\sigma^H$. Using that (jump1) and (jump2) are satisfied for $\mu^H$ and $\sigma^H$ together with (jump3) and \eqref{eqn:thm1:3} yields when $\sigma^\ast > 0$
	\begin{equation*}
		\begin{aligned}
			\bigl(\frac{\mu^H}{\sigma^H} - \frac{(\sigma^H)'}{2}\bigr)(\xi^H+) &= \mu^H(\xi^H+) = \big(H'\mu + \frac{1}{2}D^2H \cdot \sigma^2\big) (\xi+) = \big(\frac{\mu}{\sigma} - \frac{\sigma'}{2}\big) (\xi+) \\
			& \neq \big(\frac{\mu}{\sigma} - \frac{\sigma'}{2}\big) (\xi-) =  \bigl(\frac{\mu^H}{\sigma^H} - \frac{(\sigma^H)'}{2}\bigr)(\xi^H-)
		\end{aligned}
	\end{equation*}
	and analogously when $\sigma^\ast < 0$
	\[
	\bigl(\frac{\mu^H}{\sigma^H} - \frac{(\sigma^H)'}{2}\bigr)(\xi^H+) = \big(\frac{\mu}{\sigma} - \frac{\sigma'}{2}\big) (\xi-) \neq \big(\frac{\mu}{\sigma} - \frac{\sigma'}{2}\big) (\xi+) =  \bigl(\frac{\mu^H}{\sigma^H} - \frac{(\sigma^H)'}{2}\bigr)(\xi^H-).
	\]
	Thus, (jump3) holds for $\mu^H$ and $\sigma^H$. 
	
	Next, we show that $Z$ satisfies (reach jump). Using integration by substitution we have 
	\[
	p_{Z_{t^\ast}}(\xi^H) = p_{X_{t^\ast}}(H^{-1}(\xi^H)) \cdot |(H^{-1})'(\xi^H)| =  p_{X_{t^\ast}}(\xi) \cdot |\frac{1}{H'(\xi)}| > 0
	\]
	and hence (reach jump) holds for $Z$.
	
	Finally, we show that (transform) holds with $G^Z = G \circ H^{-1}$. Since $G$ and $H$ are bi-Lipschitz continuous, $G^Z$ is also bi-Lipschitz continuous. Note that since $G', H'$ are bounded absolutely continuous functions also $(G^Z)' = \frac{G'}{H'} \circ H^{-1}$ is absolutely continuous and $D^2G^Z = \frac{D^2G \cdot H' - G' \cdot D^2H}{(H')^3} \circ H^{-1}$ is a weak derivative of $(G^Z)'$. Elementary calculations can be used to show that for the transformed coefficients it holds
	\[
	\bigl((G^Z)'\mu^H + \frac{1}{2}D^2G^Z \cdot (\sigma^H)^2 \bigr) \circ (G^Z)^{-1} = \tmu \qquad \text{and} \qquad \bigl( (G^Z)' \sigma^H\bigr) \circ (G^Z)^{-1} = \tsigma.
	\]

	Hence, (transform) holds for $Z$ and so $Z$ satisfies all assumptions of Proposition~\ref{prop:lowerbound:localAdditive}. Therefore, it holds for a constant $c_1 \in (0,\infty)$, which is independent of $n \in N$,
	\begin{flalign*}
		\inf_{\substack{t_1, \dots, t_n \in [0,1] \\g \colon \mathbb{R}^n \rightarrow \mathbb{R} \: measurable}} \EE \bigl[|Z_1 - g(W_{t_1}, \dots, W_{t_n})|^2 \bigr]\geq \frac{c_1}{n^{3 \slash 4}}.
	\end{flalign*} 
	The claim now follows since $Z = H(X)$ and since $H$ is bi-Lipschitz continuous.

\end{proof}

\section{Global approximation}\label{sect:globalApprox}

In this section we prove Theorem~\ref{thm:nonadapt} and Theorem~\ref{thm:nonadapt:auton}. For this we show that for global approximations the best possible $L^p$-error rate that can be achieved by any adaptive method is at most $1/2$ under the assumptions of the theorems. First, we introduce the class of adaptive methods and thereafter we show the lower bounds.

\subsection{The class of adaptive algorithms}

Instead of studying lower bounds for methods based on finitely many evaluations of the Brownian motion $W$ as in Theorem~\ref{thm:nonadapt} and Theorem~\ref{thm:nonadapt:auton}, we later consider more general methods where the evaluation points of the Brownian motion can be chosen adaptively. As in Section 4 in~\cite{hhmg2019}, we consider sequences 
\begin{equation*}
	\psi = (\psi_k)_{k \in \N}, \qquad \chi = (\chi_k)_{k \in \N}, \qquad \varphi =  (\varphi_k)_{k \in \N}
\end{equation*}
of measurable functions
\begin{equation*}
	\begin{aligned}
		&\psi_k\colon \R^{k}\rightarrow [0,T],\\
		& \chi_k\colon \R^{k+1} \rightarrow \{\text{STOP}, \text{GO}\},\\
		& \varphi_k\colon \R^{k+1} \rightarrow L^1([0,T]).
	\end{aligned}
\end{equation*}
Here, $\psi$ is used to determine the evaluation points of $W$ and $\chi$ specifies when the evaluation of $W$ is stopped. If no further evaluations of $W$ are carried out, $\varphi$ is used to obtain the result of the approximation method.

To get a better idea of such approximation methods, we consider a realization $x_0 \in \R$ of $X_0$ and a path $w \in C([0,T])$ of $W$. In the first step, $w$ is evaluated at the point $\psi_1(x_0)$ and for $k \in \N$ we write $D_k(x_0, w) = (x_0, y_1, \dots, y_k)$, where $y_1 = w(\psi_1(x_0))$ and $y_k = w(\psi_k(D_{k-1}(x_0, w)))$, for the already observed data of $w$. Depending on $\chi_k(D_k(x_0, w))$, then further
evaluations of $w$ are carried out or not. The total number of evaluations of $w$ is then given by
\[
\nu(x_0, w) = \inf \{k \in \N \colon \chi_k(D_k(x_0, w)) = \text{STOP}\}.
\]
We require that $\nu(x_0,w) < \infty$ holds for $\PP^{(X_0, W)}$-almost all $(x_0, w) \in \R \times C([0,T])$. The approximation method is then given by
\[
\widehat{X} = \varphi_{\nu(X_0, W)}(D_{\nu(X_0, W)}(X_0, W))
\]
and for its cost we write
\[
c(\widehat{X}) = \EE[\nu(X_0, W)].
\]
We denote the class of all methods of the above form by $\mathcal{A}^{adapt}(L^1([0,T]), X_0, W)$ and we write for the class of adaptive methods with a cost of at most $n \in \N$ 
\[
\mathcal{A}^{adapt}_n(L^1([0,T]), X_0, W) = \{ \widehat{X} \in \mathcal{A}^{adapt}(L^1([0,T]), X_0, W)\colon c(\widehat{X}) \le n\}.
\]

\subsection{Proof of Theorem~\ref{thm:nonadapt}}

In the following, instead of Theorem~\ref{thm:nonadapt}, we show the more general statement of the following theorem.

	\begin{theorem}\label{thm:adapt}
		Let $T \in (0, \infty)$, $\mu, \sigma\colon [0,T] \times \R \rightarrow \R$ be measurable functions and let $X\colon [0, T] \times \Omega \rightarrow \R$ be an adapted process with continuous paths such that 
		\begin{equation}\label{sde2}
			X_t = X_0 + \int_0^t \mu(s, X_s) \, ds + \int_0^t \sigma(s, X_s) \, dW_s, \qquad t \in [0,T].
		\end{equation}
		Assume that there exist $t_0 \in [0, T), T_0 \in (t_0, T], \delta \in (0, \infty)$ and $\xi \in \R$ such that (local Lip), (non-deg) and (reach) from Theorem~\ref{thm:nonadapt} hold.
		
		Then there exists a constant $c \in (0, \infty)$ such that for all $n \in \mathbb{N}$,
		\[
		\inf_{\widehat X^n \in \mathcal{A}^{adapt}_n(L^1([0,T]), X_0, W)} \EE\big[ \|X - \widehat X^n\|_{L^1([0,T])}\big] \ge \frac{c}{n^{1/2}}.
		\]
	\end{theorem}
	
	Similar to~\cite{hhmg2019}, instead of $\Acal_n^{adapt}(L^1([0,T]), X_0, W)$ we first consider a further class of algorithms for technical reasons. In the following proposition we derive a suitable lower bound for methods from this class.

	\begin{prop}\label{prop:adapt}
		Let the assumptions of Theorem~\ref{thm:adapt} hold. Then there exist constants $c_1, c_2, c_3 \in (0, \infty)$ such that for all $n \in 2\N$, all random variables $\widehat{X}^n\colon \Omega \rightarrow L^1([0,T])$, all random variables $\tau_0, \dots, \tau_{3n/2}\colon \Omega \rightarrow [0, T]$ with
		\begin{itemize}
			\item [(B1)] $\tau_l$ is measurable with respect to $\sigma(\Fc_0, W_{\tau_0}, \dots, W_{\tau_{l-1}})$ for $l \in \{1, \dots, 3n/2\}$,
			\item [(B2)] $\tau_l = t_0 + (T_0 - t_0)l/n$ for $l \in \{0, \dots, n\}$,
			\item [(B3)] $\widehat{X}^n$ is measurable with respect to $\mathfrak A = \sigma(\Fc_0, W_{\tau_0} , \dots, W_{\tau_{3n/2}})$,
		\end{itemize}
		and all $A \in \mathfrak{A}$ it holds
		\[
		\EE\big[ 1_A\|X - \widehat{X}^n \|_{L^1([0,T])}\big] \ge c_1 \cdot (c_2 - \PP(A^\mathsf{c}))\cdot n^{-1/2} - \frac{c_3}{n}.
		\]
	\end{prop}
	\begin{proof}
		We use ideas from the proof of Proposition~5 in~\cite{hhmg2019} and we may assume $t_0 = 0$ with similar arguments as in that proof. Throughout this proof we use $c_1,c_2,\dots\in (0,\infty)$ to denote positive constants which do not depend on $n$.
		
		Due to (reach), there exist $\delta^\ast, \delta_0, \delta_1, \delta_2 \in (0, \delta)$ such that $\delta^\ast < \delta_0 < \delta_1 < \delta_2 < \delta$ and 
		\[
		\PP(X_0 \in B_{\delta^\ast}(\xi)) > 0.
		\]
		
		Let $\eta_1, \eta_2\colon \R \rightarrow \R$ be infinitely times differentiable functions with $\eta_1, \eta_2 \in [0,1]$ such that for $x \in \R$ it holds
		\[
		\eta_1(x) = \begin{cases}
			1, &\text{if}\, x \in B_{\delta_1}(\xi), \\
			0, &\text{if}\, x \notin B_{\delta_2}(\xi),
		\end{cases} 
		\qquad 
		\eta_2(x) = \begin{cases}
			0, &\text{if}\, x \in B_{\delta_0}(\xi), \\
			1, &\text{if}\, x \notin B_{\delta_1}(\xi).
		\end{cases} 
		\]		
		Moreover, let $\mu^\ast, \sigma^\ast\colon [0,T_0] \times \R \rightarrow \R$ be given by
		\[
		\mu^\ast(t,x) = \eta_1(x) \cdot \mu(t,x), \quad \sigma^\ast(t,x) = \eta_1(x) \cdot \sigma(t,x) + \eta_2(x) \cdot \sgn(\sigma(t,x)), \quad (t, x) \in  [0,T_0] \times \R.
		\]
		
		Because of (local Lip) and (non-deg) the functions $\mu^\ast, \sigma^\ast$ are bounded, Lipschitz continuous and 
		\begin{equation}\label{eqn_adapt_3}
			\inf_{(t,x) \in [0,T_0] \times \R} |\sigma^\ast(t,x)|  > 0.
		\end{equation}
		
		Let $X^\ast$ be a solution of the SDE~\eqref{sde2} on the interval $[0, T_0]$ with drift $\mu^\ast$, diffusion $\sigma^\ast$ and initial value $X_0$.
		
		Let $n \in 2\N$,  $\widehat{X}^n\colon \Omega \rightarrow L^1([0,T])$ be a random variable, $\tau_0, \dots, \tau_{3n/2}\colon \Omega \rightarrow [0,T]$ be random variables such that (B1)-(B3) hold and let $A \in \mathfrak{A}$.
		
		In the following we use ideas of the proof of Proposition 10 in~\cite{hhmg2019}.
		We set $t_l = T_0l/n$ for $l \in \{0, \dots, n\}$ and we set $\overline{X}^{\ast,n}_0 = X_0$ and for $l \in \{1, \dots, n\}$, $t \in (t_{l-1}, t_{l}]$
		\[
		\overline{X}^{\ast,n}_t = X^{\ast}_{t_{l-1}} + \sigma^\ast( t_{l-1}, X^{\ast}_{t_{l-1}}) \cdot (W_t - W_{t_{l-1}}).
		\]

		We set for $\gamma \in (0, \delta]$ and $0 < \widehat{T} \le T$
		\[
		\widehat \tau^{[0, \widehat{T}], \gamma}\colon C([0, \widehat{T}]) \rightarrow [0, \widehat{T}], \quad f \mapsto \inf\{s \in [0, \widehat{T}] \colon f(s) \notin B_{\gamma}(\xi)\}  \wedge \widehat{T}.
		\]
		
		Since $\mu^\ast, \sigma^\ast$ are bounded and $\sigma^\ast$ is Lipschitz continuous, it holds for all $l \in \{1, \dots, n\}$ and all $s \in (t_{l-1}, t_l]$
		\begin{equation*}
			\begin{aligned}
				&\EE\Big[ \big|X^\ast_s - \overline{X}^{\ast, n}_s\big| \Big] = \EE\Big[ \big|\int_{t_{l-1}}^s \mu^\ast(s, X^\ast_s) \, ds + \int_{t_{l-1}}^s \sigma^\ast(s, X^\ast_s) - \sigma^\ast(t_{l-1}, X^\ast_{t_{l-1}})\, dW_s \big| \Big] \\
				& \qquad \qquad \le \frac{c_1}{n} + c_2\Big(\int_{t_{l-1}}^s \frac{1}{n^2} + \EE \big[| X^\ast_s - X^\ast_{t_{l-1}}|^2 \big] \, ds\Big)^{1/2} \le \frac{c_3}{n}.
			\end{aligned}
		\end{equation*}
		Hence, it holds with Lemma 20 in~\cite{hhmg2019}
		\begin{equation}\label{eqn_adapt_1}
			\begin{aligned}
				&\EE\big[ 1_A \|X - \widehat{X}^n \|_{L^1([0,T])}\big] \\
				&\qquad \geq \EE\big[ 1_A \|X - \widehat{X}^n \|_{L^1([0,T_0])}\big] \\
				& \qquad = \sum_{l = 1}^n \int_{t_{l-1}}^{t_l} \EE \big[1_A|X_s - \widehat{X}^n_s|\big] \, ds \\
				& \qquad \ge \sum_{l= 1}^n \int_{t_{l-1}}^{t_l} \EE \big[1_A \cdot 1_{\{\widehat \tau^{[0, t_l], \delta_0}(X^\ast) = t_l\}}|X^\ast_s - \widehat{X}^n_s|\big] \, ds \\
				& \qquad \ge \sum_{l= 1}^n \int_{t_{l-1}}^{t_l} \EE \big[1_A \cdot 1_{\{\widehat \tau^{[0, t_l], \delta_0}(X^\ast) = t_l\}}|\overline{X}^{\ast, n}_s -  \widehat{X}^n_s|\big] \, ds - \int_0^{T_0} \EE\big[ |X^\ast_s - \overline{X}^{\ast, n}_s|\big] \, ds \\
				& \qquad \ge \sum_{l= 1}^n \int_{t_{l-1}}^{t_l} \EE \big[1_A \cdot 1_{\{\widehat \tau^{[0, t_{l-1}], \delta_0}(X^\ast) = t_{l-1}\} \cap \{\widehat \tau^{[0, t_l - t_{l-1}], \delta_0}(X^\ast_{\cdot + t_{l-1}}) = (t_l - t_{l-1}) \}}|\overline{X}^{\ast, n}_s -  \widehat{X}^n_s|\big] \, ds - \frac{c_4}{n}.
			\end{aligned}
		\end{equation}
		
		Since $\mu^\ast, \sigma^\ast$ are Lipschitz continuous there exist, due to Theorem 1 in~\cite{Ka96}, for $l \in \{1, \dots, n\}$ functions $F^\ast_{t_{l-1}} \colon \R \times C([0, t_{l-1}]) \rightarrow C([0, t_{l-1}])$, $F^{\ast}_{T_0/n}\colon  \R \times C([0, T_0/n]) \rightarrow C([0, T_0/n])$ such that
		\begin{equation}\label{eqn_adapt_2}
			\begin{aligned}
				(X^\ast_s)_{s \in [0, t_{l-1}]} &= F^\ast_{t_{l-1}}(X_0, (W_u)_{u \in [0, t_{l-1}]}), \\ (X^\ast_{s + t_{l-1}})_{s \in [0, t_l - t_{l-1}]} &= F^\ast_{T_0/n}(X_{t_{l-1}}, (W_{u + t_{l-1}} - W_{t_{l-1}})_{u \in [0, t_l - t_{l-1}]}).
			\end{aligned}
		\end{equation}

		To be able to order $\tau_0, \dots, \tau_{3n/2}$ we also introduce $\tau^{ord}_0, \dots, \tau^{ord}_{3n/2}$ with $\tau^{ord}_{0} \le \tau^{ord}_1 \le \dots \le \tau^{ord}_{3n/2}$ and $\{\tau_0, \dots, \tau_{3n/2}\} = \{\tau^{ord}_0, \dots, \tau^{ord}_{3n/2}\}$. Therewith, we define the piecewise linear interpolation of $W$ for $l \in \{1, \dots, 3n/2\}$ with $\tau^{ord}_{l-1} < \tau^{ord}_l$ and $s \in [\tau^{ord}_{l-1}, \tau^{ord}_l]$ by
		\[
		\overline{W}_{s} = \frac{\tau^{ord}_l - s}{\tau^{ord}_l - \tau^{ord}_{l-1}} W_{\tau^{ord}_{l-1}} + \frac{s - \tau^{ord}_{l-1}}{\tau^{ord}_l - \tau^{ord}_{l-1}} W_{\tau^{ord}_l} = W_{\tau^{ord}_{l-1}} + \frac{s - \tau^{ord}_{l-1}}{\tau^{ord}_l - \tau^{ord}_{l-1}} \big( W_{\tau^{ord}_l} - W_{\tau^{ord}_{l-1}} \big)
		\] 
		and 
		\[
		B = W - \overline{W}.
		\]
		Then it holds
		\begin{itemize}[align=left,labelwidth=\widthof{(given $\mathfrak{A}$)},leftmargin=\labelwidth+\labelsep]
			\item [(given $\mathfrak{A}$)] conditioned on $\mathfrak{A}$,
			\begin{itemize}
				\item [($\mathfrak{A}$-1)] the values  $X_0, \tau_0, \dots, \tau_{3n/2}, \tau^{ord}_0, \dots, \tau^{ord}_{3n/2}$ and the processes $\overline{W}$, $\widehat{X}^n$ are fixed,
				\item [($\mathfrak{A}$-2)] $B$ consists of Brownian bridges on each of the intervals $[\tau^{ord}_{l-1},\tau^{ord}_l]$ for \\$l \in \{1, \dots, 3n/2\}$ with $\tau^{ord}_{l-1} < \tau^{ord}_l$ which are independent,
			\end{itemize}
		\end{itemize} 
		cf. Lemma 1 and Lemma 2 in~\cite{Y17}.
		Setting for $l \in \{1, \dots, n\}$
		\begin{equation}\label{eqn_adapt_5}
			\begin{aligned}
				(\tX^{\ast, l}_{s + t_{l-1}})_{s \in [0, t_l - t_{l-1}]} &= F^\ast_{T_0/n}(X_{t_{l-1}}, (\overline{W}_{u + t_{l-1}} - B_{u + {t_{l-1}}} - \overline{W}_{t_{l-1}})_{u \in [0, t_l - t_{l-1}]}),
				\\\overline{\tX}^{\ast,n}_s& = X^{\ast}_{t_{l-1}} + \sigma^\ast(t_{l-1}, X^{\ast}_{t_{l-1}}) \cdot (\overline{W}_s- B_s - \overline{W}_{t_{l-1}}), \qquad \qquad s \in (t_{l-1}, t_l], 
			\end{aligned}
		\end{equation}
		we thus get for $i \in \{1, \dots, n/2\}$ by \eqref{eqn_adapt_2}, (given $\mathfrak{A}$) and since $\widehat \tau^{[0, t_{l-1}], \delta_0}, \widehat \tau^{[0, t_{l} - t_{l-1}], \delta_0}$ are measurable
		\begin{equation*}
			\begin{aligned}
				&\EE \big[\int_{t_{l-1}}^{t_{l}} 1_{\{\widehat \tau^{[0, t_{l-1}], \delta_0}(X^\ast) = t_{l-1}\} \cap \{\widehat \tau^{[0, t_{l} - t_{l-1}], \delta_0}(\tX^{\ast, l}_{\cdot + t_{l-1}}) = (t_{l} - t_{l-1}) \}}|\overline{\tX}^{\ast, n}_s -  \widehat{X}^n_s| \, ds \,\vert \mathfrak{A}\big]\\
				&\qquad  = \EE \big[ \int_{t_{l-1}}^{t_{l}} 1_{\{\widehat \tau^{[0, t_{l-1}], \delta_0}(X^\ast) = t_{l-1}\} \cap \{\widehat \tau^{[0, t_{l} - t_{l-1}], \delta_0}(X^\ast_{\cdot + t_{l-1}}) = (t_{l} - t_{l-1}) \}}|\overline{X}^{\ast, n}_s -  \widehat{X}^n_s| \, ds \,\vert \mathfrak{A}\big].
			\end{aligned}
		\end{equation*}
		Let $l \in \{1, \dots, n\}$. Therefore, we obtain with
		\[
		A_{l} = \{\widehat \tau^{[0, t_{l-1}], \delta_0}(X^\ast) = t_{l-1}\} \cap \{\widehat \tau^{[0, t_{l} - t_{l-1}], \delta_0}(X^\ast_{\cdot + t_{l-1}}) = \widehat \tau^{[0, t_{l} - t_{l-1}], \delta_0}(\tX^{\ast, l}_{\cdot + t_{l-1}}) = (t_{l} - t_{l-1}) \}
		\]
		the validity of
		\begin{equation*}
			\begin{aligned}
				&\EE \big[\int_{t_{l-1}}^{t_{l}} 1_{A_{l}} |\overline{X}^{\ast, n}_s - \overline{\tX}^{\ast, n}_s| \, ds \,\vert \mathfrak{A}\big] \\
				& \qquad \le 2 \EE \big[ \int_{t_{l-1}}^{t_{l}} 1_{\{\widehat \tau^{[0, t_{l-1}], \delta_0}(X^\ast) = t_{l-1}\} \cap \{\widehat \tau^{[0, t_{l} - t_{l-1}], \delta_0}(X^\ast_{\cdot + t_{l-1}}) = (t_{l} - t_{l-1}) \}}|\overline{X}^{\ast, n}_s -  \widehat{X}^n_s| \, ds \,\vert \mathfrak{A}\big],
			\end{aligned}
		\end{equation*}
		which gives us with \eqref{eqn_adapt_3} and the definitions of $\overline{X}^{\ast, n}, \overline{\tX}^{\ast, n}$ 
		\begin{equation}
			\begin{aligned}\label{eqn_adapt_4}
				&\EE \big[ \int_{t_{l-1}}^{t_{l}} 1_{\{\widehat \tau^{[0, t_{l-1}], \delta_0}(X^\ast) = t_{l-1}\} \cap \{\widehat \tau^{[0, t_{l} - t_{l-1}], \delta_0}(X^\ast_{\cdot + t_{l-1}}) = (t_{l} - t_{l-1}) \}}|\overline{X}^{\ast, n}_s -  \widehat{X}^n_s| \, ds \,\vert \mathfrak{A}\big] \\
				& \qquad \qquad \ge  c_5 \EE \big[\int_{t_{l-1}}^{t_{l}} 1_{A_{l}} |B_s| \, ds \,\vert \mathfrak{A}\big].
			\end{aligned}
		\end{equation}
		Now we have due to \eqref{eqn_adapt_2}, \eqref{eqn_adapt_5}, (given $\mathfrak{A}$) and the measurability of $\widehat \tau^{[0, t_{l-1}], \delta^\ast},\widehat \tau^{[0, t_{l-1}], \delta_0}$ and $ \widehat \tau^{[0, t_{l} - t_{l-1}], \delta_0}$
		\begin{equation*}
			\begin{aligned}
				&\EE \big[\int_{t_{l-1}}^{t_{l}} 1_{A_{l}} |B_s| \, ds \,\vert \mathfrak{A}\big] \\
				& \qquad \qquad \ge \EE \big[\int_{t_{l-1}}^{t_{l}} 1_{A_{l} \cap \{|B_s| \le 1\}}  |B_s| \, ds \,\vert \mathfrak{A}\big] \\
				& \qquad \qquad \ge \EE\big[\int_{t_{l-1}}^{t_{l}}1_{\{\widehat \tau^{[0, t_{l-1}], \delta^\ast}(X^\ast) = t_{l-1}\} \cap \{|B_s| \le 1\}} |B_s| \, ds \,\vert \mathfrak{A} \big] \\
				& \qquad \qquad \qquad \qquad - 2 \EE \big[1_{\{X^\ast_{t_{l-1}} \in \overline{B_{\delta^\ast}(\xi)}\} \cap \{\widehat \tau^{[0, t_{l} - t_{l-1}], \delta_0}(X^\ast_{\cdot + t_{l-1}}) < (t_{l} - t_{l-1}) \}} \,\vert \mathfrak{A}\big] \\
				& \qquad \qquad \ge \EE\big[\int_{t_{l-1}}^{t_{l}}1_{\{\widehat \tau^{[0, t_{l-1}], \delta^\ast}(X^\ast) = t_{l-1}\}} |B_s| \, ds \,\vert \mathfrak{A} \big] - \frac{c_6}{n^2} \\
				& \qquad \qquad \qquad \qquad - 2 \EE \big[1_{\{X^\ast_{t_{l-1}} \in \overline{B_{\delta^\ast}(\xi)}\} \cap \{\widehat \tau^{[0, t_{l} - t_{l-1}], \delta_0}(X^\ast_{\cdot + t_{l-1}}) < (t_{l} - t_{l-1}) \}} \,\vert \mathfrak{A}\big]
			\end{aligned}
		\end{equation*}
		Combining this with \eqref{eqn_adapt_1}, \eqref{eqn_adapt_4} and using that 
		\begin{equation*}
			\begin{aligned}
				&\PP(\{X^\ast_{t_{l-1}} \in \overline{B_{\delta^\ast}(\xi)}\} \cap \{\widehat \tau^{[0, t_{l} - t_{l-1}], \delta_0}(X^\ast_{\cdot + t_{l-1}}) < (t_{l} - t_{l-1}) \}) \\
				& \qquad \qquad \le \PP(\{|X^\ast_{\widehat \tau^{[0, t_{l} - t_{l-1}], \delta_0}(X^\ast_{\cdot + t_{l-1}}) \wedge (t_l - t_{l-1})} - X^\ast_{t_{l-1}}| \ge (\delta_0 - \delta^\ast)\}) \\
				& \qquad \qquad \le \frac{c_7}{n^2}
			\end{aligned}
		\end{equation*}
		shows, since $A \in \mathfrak{A}$, that
		\begin{equation}\label{eqn_adapt_6}
			\EE\big[ 1_A \|X - \widehat{X}^n \|_{L^1([0,T])}\big] \ge c_5 \EE\Big[ 1_A \EE \big[ \sum_{l=1}^n \int_{t_{l-1}}^{t_{l}}1_{\{\widehat \tau^{[0, t_{l-1}], \delta^\ast}(X^\ast) = t_{l-1}\}} |B_s| \, ds \, \vert \mathfrak{A}\big] \Big]- \frac{c_8}{n}.
		\end{equation}
		For $l \in \{0, \dots, n-1\}$ we set $d_l = \#\{i \in \{0, \dots, 3n/2\} \colon \tau_i \in (t_l, t_{l+1})\}$. Let $l_1, \dots, l_{n/2} \in \{0, \dots, n-1\}$ with $l_1 < \dots < l_{n/2}$ and $d_{l_i} = 0$ for all $i \in \{1, \dots, n/2\}$. Note that for any Brownian bridge $B^{s,t}$ on the interval $[s,t]$ with $0 \le s < t \le \max\{1,T\}$ it holds by the scaling property of Brownian bridges
		\[
		\EE \bigl[\int_s^t |B^{s,t}_u| \, du \bigr] = (t-s) \EE \bigl[\int_0^1 |B^{s,t}_{s + (t-s)u}| \, du \bigr] = (t-s)^{3/2}\EE \bigl[\int_0^1 |B^{0,1}_u| \, du\bigr].
		\] 
		Thus, it holds by (given $\mathfrak{A}$), \eqref{eqn_adapt_2} and \eqref{eqn_adapt_6}
		\begin{equation*}
			\begin{aligned}
				\EE\big[ 1_A \|X - \widehat{X}_n \|_{L^1([0,T])}\big] &\ge c_5 \EE\Big[1_A \EE\big[ \sum_{i=1}^{n/2} \int_{t_{l_i}}^{t_{l_{i+1}}}1_{\{\widehat \tau^{[0, t_{l_i}], \delta^\ast}(X^\ast) = t_{l_i}\}} |B_s| \, ds \, \vert \mathfrak{A}\big] \Big] - \frac{c_8}{n} \\
				&\ge c_9 \EE\Big[ 1_A \sum_{i=1}^{n/2}(t_{l_{i+1}} - t_{l_i})^{3/2}\EE\big[  1_{\{\widehat \tau^{[0, t_{l_i}], \delta^\ast}(X^\ast) = t_{l_i}\}} \, \vert \mathfrak{A}\big] \Big] - \frac{c_8}{n} \\
				&\ge \frac{c_{10}}{n^{1/2}} \EE \big[ 1_A \cdot \PP(\{\forall s \in [0,T_0] \colon X^\ast_s \in B_{\delta^\ast}(\xi)\} \, \vert \mathfrak{A}) \big] - \frac{c_8}{n}\\
				&\ge \frac{c_{10}}{n^{1/2}} \Big( \PP(\{\forall s \in [0,T_0] \colon X^\ast_s \in B_{\delta^\ast}(\xi)\}) - \PP(A^{\mathsf{c}}) \Big) - \frac{c_8}{n}.
			\end{aligned}
		\end{equation*}
		Since $\mu^\ast, \sigma^\ast$ are bounded and Lipschitz continuous and since (reach) as well as \eqref{eqn_adapt_3} hold, the claim follows with a support theorem of Pakkanen, see \cite[Theorem 3.2]{Pakkanen2010}.
	\end{proof}
	
	Now we are ready to show Theorem~\ref{thm:adapt}. To do this, we construct for a method $\widehat X^n$ of the class $\Acal_n^{adapt}(L^1([0,T]), X_0, W)$ a new method that satisfies (B1)-(B3) and apply Proposition~\ref{prop:adapt} afterwards.
	
	\begin{proof}[Proof of Theorem~\ref{thm:adapt}]
		We subsequently use ideas from the proof of Theorem 6 in~\cite{hhmg2019}.
		Let $k \in \N$ and $\widehat{X}^{adapt,k} \in \Acal_k^{adapt}(L^1([0,T]), X_0, W)$. Let $c_1, c_2, c_3 \in (0, \infty)$ be as in Proposition~\ref{prop:adapt} and let $m = \lceil \frac{2k}{c_2} \rceil$. Then it holds
		\[
		\PP(\nu(X_0, W) \ge m) \le \frac{\EE[\nu(X_0, W)]}{m} \le \frac{c_2}{2}.
		\]
		Let $n = 2m$ and define $\widehat{X}^{n}$ by
		\[
		\widehat{X}^n
		= \begin{cases}
			\widehat{X}^{adapt,k}, &\text{if $\nu(X_0, W) < n/2$,}\\
			0 ,& \text{otherwise}.
		\end{cases}
		\]
		Then the map $\widehat{X}^n\colon \Omega \rightarrow L^1([0,T])$ is a random variable and there exist random variables $\tau_0, \dots, \tau_{3n/2}\colon \Omega \rightarrow [0, T]$ such that (B1)-(B3) from Proposition~\ref{prop:adapt} and $\{\nu(X_0, W) < n/2\} \in \mathfrak{A}$ hold.  
		
		Because of
		\begin{equation*}
			\begin{aligned}
				\EE\big[ \|X - \widehat{X}^{adapt,k} \|_{L^1([0,T])}\big] &\ge \EE\big[ 1_{\{\nu(X_0, W) < n/2\}}\|X - \widehat{X}^{adapt,k} \|_{L^1([0,T])}\big] \\
				& = \EE\big[ 1_{\{\nu(X_0, W) < n/2\}}\|X - \widehat{X}^n \|_{L^1([0,T])}\big],
			\end{aligned}
		\end{equation*}
		the claim follows with Proposition~\ref{prop:adapt}.
	\end{proof}

\subsection{Proof of Theorem~\ref{thm:nonadapt:auton}}
	Similar to the previous section, we show that any sequence of adaptive methods has an $L^p$-error rate of at most $1/2$ under the assumptions of Theorem~\ref{thm:nonadapt:auton}. The following theorem implies in particular Theorem~\ref{thm:nonadapt:auton}.
	
	\begin{theorem}\label{thm:adapt:auton}
		Let $T \in (0, \infty)$, $\mu, \sigma\colon \R \rightarrow \R$ be measurable functions and let $X \colon [0, T] \times \Omega \rightarrow \R$ be an adapted process with continuous paths such that 
		\begin{equation*}
			X_t = X_0 + \int_0^t \mu(X_s) \, ds + \int_0^t \sigma(X_s) \, dW_s, \qquad t \in [0,T].
		\end{equation*}
		Assume that there exist $t_0 \in [0, T), T_0 \in (t_0, T], \delta \in (0, \infty)$ and $\xi \in \R$ such that (reach), (local~reg) and (non-deg*) from Theorem~\ref{thm:nonadapt:auton} hold.
		
		Then there exists a constant $c \in (0, \infty)$ such that for all $n \in \mathbb{N}$,
		\[
		\inf_{\widehat X^n \in \Acal_n^{adapt}(L^1([0,T]), X_0, W)} \EE\big[ \|X - \widehat X^n\|_{L^1([0,T])}\big] \ge \frac{c}{n^{1/2}}.
		\]
	\end{theorem}
	
	For the proof of the above theorem, we use the fact that the coefficients $\mu, \sigma$ coincide locally with other coefficients $\mu^\ast, \sigma^\ast$ that satisfy (transform). The exact definitions of $\mu^\ast$ and $\sigma^\ast$ can be seen in the following lemma.
	
	\begin{lemma}\label{gE:loc:transform}
		Assume that there exist $t_0 \in [0, T), T_0 \in (t_0, T], \delta \in (0, \infty)$ and $\xi \in \R$ such that (local reg) and (non-deg*) hold.
		Let
		\[
		\sigma^\ast = \sigma(\xi - \delta) 1_{(-\infty, \xi - \delta)} + \sigma 1_{[\xi-\delta, \xi + \delta]} + \sigma(\xi + \delta) 1_{(\xi + \delta, \infty)}
		\]
		denote the constant continuation of the coefficient $\sigma$ and let
		\[
		\mu^\ast = 1_{[\xi-\delta, \xi+\delta]} \mu.
		\] 
		Then the coefficients $\mu^\ast, \sigma^\ast$ satisfy (transform).
	\end{lemma}
	
	\begin{proof}
		As in the proof of Theorem~\ref{thm:lowebound} we consider the Lamperti-type transform
		\[
		H\colon \R \rightarrow \R, \qquad x \mapsto \int_0^x \frac{1}{\sigma^\ast(z)} \, dz.
		\]
		Similar to the proof of Theorem~\ref{thm:lowebound} one can show that $H$ is differentiable with absolutely continuous derivative $H' = \frac{1}{\sigma^\ast}$ and similar to \eqref{eqn:thm1:3}
		\[
		D^2H = -1_{B_\delta(\xi)}\frac{ \delta_{\sigma}}{\sigma^2}
		\]
		is a weak derivative of $H'$ where $\delta_\sigma(x) = \sigma'(x)$ if $\sigma$ is differentiable in $x$ and $\delta_\sigma(x) = 0$ otherwise for $x \in \R$.	In consideration of (local~reg) and (non-deg*), we therefore obtain that the transformed coefficient $(\mu^\ast)^H = \bigl(H'\mu^\ast + \frac{1}{2}D^2H \cdot (\sigma^\ast)^2 \bigr) \circ H^{-1}$ is a bounded integrable function and we have $(\sigma^\ast)^H = \bigl( H' \sigma^\ast\bigr) \circ H^{-1} = 1$. With Lemma~1 and Lemma~2 in~\cite{EMGY25Hoelder} we see that $(\mu^\ast)^H$ and $(\sigma^\ast)^H$ satisfy (transform). With similar arguments as in the proof of Theorem~\ref{thm:lowebound} also $\mu^\ast$ and $\sigma^\ast$ satisfy (transform).
	\end{proof}
	
	By applying a suitable transformation, Theorem~\ref{thm:adapt:auton} now follows with Theorem~\ref{thm:adapt}.

	\begin{proof}[Proof of Theorem~\ref{thm:adapt:auton}]
		Let $\mu^\ast, \sigma^\ast$ be as in Lemma~\ref{gE:loc:transform}. Then by Lemma~\ref{gE:loc:transform} the coefficients $\mu^\ast, \sigma^\ast$ satisfy (transform) with a bi-Lipschitz continuous transformation $G^\ast \colon \R \rightarrow \R$ and a weak derivative $D^2G^\ast$ of $(G^\ast)'$. The goal is to apply Theorem~\ref{thm:adapt} to the process $Z = G^\ast(X)$ and then the claim follows. Therefore, we prove that the assumptions of Theorem~\ref{thm:adapt} are fulfilled.
		
		By the It\^{o}-formula, see e.g.~\cite[Problem 3.7.3]{ks91}, it holds
		\[
		Z_t = Z_0 + \int_0^t \mu^{G^\ast}(X_s) \, ds + \int_0^t \sigma^{G^\ast}(X_s) \, dW_s, \qquad t \in [0,T],
		\]
		where $\mu^{G^\ast} = \big( (G^\ast)'\mu + \frac{1}{2}D^2G^\ast \cdot \sigma^2 \big) \circ (G^\ast)^{-1}$ and $\sigma^{G^\ast} = ((G^\ast)' \sigma) \circ (G^\ast)^{-1}$.
		
		Since $G^\ast$ is bi-Lipschitz continuous, there exist $\xi^\ast \in \R$ and $\delta^\ast \in (0, \infty)$ such that $G^\ast(B_\delta(\xi)) = B_{\delta^\ast}(\xi^\ast)$.
		
		Now by (transform) and by the Lipschitz continuity of $G^\ast$ the functions $\big( (G^\ast)'\mu^\ast + \frac{1}{2}D^2G^\ast \cdot (\sigma^\ast)^2 \big)$ and $(G^\ast)' \sigma^\ast$ are Lipschitz continuous and hence, by the choice of $\mu^\ast$ and $\sigma^\ast$, $\big( (G^\ast)'\mu + \frac{1}{2}D^2G^\ast \cdot \sigma^2 \big)$ and $(G^\ast)' \sigma$ are Lipschitz continuous on $[\xi-\delta, \xi+\delta]$. Since $G^\ast(B_\delta(\xi)) = B_{\delta^\ast}(\xi^\ast)$, thus $\mu^{G^\ast}$ and $\sigma^{G^\ast}$ are Lipschitz continuous on $[\xi^\ast - \delta^\ast, \xi^\ast + \delta^\ast]$. So, (local Lip) is satisfied. 
		
		Also since $G^\ast(B_\delta(\xi)) = B_{\delta^\ast}(\xi^\ast)$, $G^\ast$ is bi-Lipschitz continuous and (non-deg*) holds,
		\[
		\inf_{x \in B_{\delta^\ast}(\xi^\ast)} |\sigma^{G^\ast}(x)| > 0
		\]
		and therefore (non-deg) is fulfilled.
		
		Moreover, (reach) holds since, because of $G^\ast(B_\delta(\xi)) = B_{\delta^\ast}(\xi^\ast)$, $\PP(Z_{t_0} \in B_{\delta^\ast}(\xi^\ast)) = \PP(X_{t_0} \in B_\delta(\xi))$.

		By the bi-Lipschitz continuity of $G^\ast$ there now exists a constant $c_1 \in (0, \infty)$, which is independent of $n$, such that
		\begin{equation*}
			\begin{aligned}
				&\inf_{\widehat X^n \in \Acal_n^{adapt}(L^1([0,T]), X_0, W)} \EE\big[ \|X - \widehat X^n\|_{L^1([0,T])}\big] \\
				& \qquad \qquad \ge c_1 \inf_{\widehat X^n \in \Acal_n^{adapt}(L^1([0,T]), X_0, W)} \EE\big[ \|Z - G^\ast(\widehat X^n)\|_{L^1([0,T])}\big].
			\end{aligned}
		\end{equation*}
		Since $G^\ast$ is Lipschitz continuous, it satisfies the linear growth property and therefore for any $f \in L^1([0,T])$ it holds $G^\ast(f) \in L^1([0,T])$. Hence, the claim follows with Theorem~\ref{thm:adapt}.
	\end{proof}

\section*{Appendix}
Similar to Lemma 20 in~\cite{hhmg2019}, we show a comparison result for locally regular coefficients.

\begin{lemma}\label{lem:localJump:stopped}
	Assume that $\mu, \sigma, \mu^\ast, \sigma^\ast \colon \R \rightarrow \R$ are measurable functions such that $\mu^\ast, \sigma^\ast$ satisfy (transform) with transformation $G^\ast \colon \R \rightarrow \R$ and weak derivative $D^2G^\ast$ of $(G^\ast)'$. Let $I \subset \R$ be an open interval and assume that
	\[
	\mu^\ast(x) = \mu(x) \qquad \text{and} \qquad \sigma^\ast(x) = \sigma(x) \qquad \text{for $x \in I$}.
	\]
	Assume further that $T \in (0,\infty)$, $V = (V_t)_{t \in [0,T]}$ is a Brownian motion and that $X, X^\ast$ are adapted processes with continuous paths satisfying
	\begin{equation*}
		\begin{aligned}
			X_t &= X_0 + \int_0^t \mu(X_s) \, ds + \int_0^t \sigma(X_s) \, dV_s, \qquad \qquad &t \in [0,T], \\
			X^\ast_t &= X_0 + \int_0^t \mu^\ast(X^\ast_s) \, ds + \int_0^t \sigma^\ast(X^\ast_s) \, dV_s, \qquad \qquad &t \in [0,T].
		\end{aligned}
	\end{equation*}
	Set
	\[
	\tau := \inf\{s \in [0, T] \colon X_s \notin I\} \wedge \inf\{s \in [0, T] \colon X^\ast_s \notin I\} \wedge T.
	\]
	Then $\PP$-almost surely for all $t \in [0,T]$
	\[
	X(t \wedge \tau) = X^\ast(t \wedge \tau).
	\]
	Moreover,
	\[
	\PP(\{\forall t \in [0,T] \colon X^\ast_t = X_t\} \cap \{\forall t \in [0,T] \colon X^\ast_t \in I\}) = \PP( \{\forall t \in [0,T] \colon X^\ast_t \in I\}) .
	\]
\end{lemma}

\begin{proof}
	To prove the statement, we will first transform the solutions $X, X^\ast$ to solutions of SDEs with Lipschitz continuous coefficients and then apply Lemma~20 in~\cite{hhmg2019}.
	
	We set for the by $G^\ast$ transformed coefficients $\widetilde{\mu}^{G^\ast}=\bigl((G^\ast)' \mu  + \frac{1}{2} D^2G^\ast \cdot \sigma^2\bigr) \circ (G^\ast)^{-1}$ as well as $\widetilde{\sigma}^{G^\ast}= \bigl((G^\ast)' \sigma \bigr) \circ (G^\ast)^{-1}$ and by (transform) we have the transformed coefficients $\widetilde{\mu^\ast}=\bigl((G^\ast)' \mu^\ast  + \frac{1}{2} D^2G^\ast \cdot (\sigma^\ast)^2\bigr) \circ (G^\ast)^{-1}$ as well as $\widetilde{\sigma^\ast}= \bigl((G^\ast)' \sigma^\ast \bigr) \circ (G^\ast)^{-1}$.
	
	Now the transformed solution $Z = G^\ast(X)$ satisfies by the It\^{o} formula, see e.g.~\cite[Problem~3.7.3]{ks91},
	\[
	Z_t = G^\ast(X_0) + \int_0^t \widetilde{\mu}^{G^\ast}(Z_s) \, ds + \int_0^t \widetilde{\sigma}^{G^\ast}(Z_s) \, dV_s, \qquad \qquad t \in [0,T],
	\]
	and similarly we obtain for $Z^\ast = G^\ast(X^\ast)$ 
	\[
	Z^\ast_t = G^\ast(X_0) + \int_0^t \widetilde{\mu^\ast}(Z^\ast_s) \, ds + \int_0^t \widetilde{\sigma^\ast}(Z^\ast_s) \, dV_s, \qquad \qquad t \in [0,T].
	\]
	Next, we want to apply Lemma~20 in~\cite{hhmg2019} and we therefore show that its assumptions are satisfied. By (transform), $\widetilde{\mu^\ast}$ and $\widetilde{\sigma^\ast}$ are Lipschitz continuous.  Since $\mu^\ast(z) = \mu(z)$ and $\sigma^\ast(z) = \sigma(z)$ for all $z \in I$, we obtain by the bi-Lipschitz continuity of $G^\ast$ that $\widetilde{\mu}^{G^\ast}(x) = \widetilde{\mu^\ast}(x)$ and $\widetilde{\sigma}^{G^\ast}(x) = \widetilde{\sigma^\ast}(x)$ for all $x \in G^\ast(I)$. Since $G^\ast$ is bi-Lipschitz continuous, $I^\ast = G^\ast(I)$ is again an open interval and for
	\[
	\tau^{G^\ast} := \inf\{s \in [0, T] \colon Z_s \notin I^\ast\} \wedge \inf\{s \in [0, T] \colon Z^\ast_s \notin I^\ast\} \wedge T 
	\]  
	it holds
	\[
	\tau^{G^\ast}  = \tau.
	\]
	Thus, it holds with Lemma~20 in~\cite{hhmg2019} $\PP$-almost surely for all $t \in [0, T]$
	\[
	Z(t \wedge \tau) = Z(t \wedge \tau^{G^\ast}) = Z^\ast(t \wedge \tau^{G^\ast}) = Z^\ast(t \wedge \tau)
	\]
	and 
	\[
	\PP(\{\forall t \in [0,T] \colon Z^\ast_t = Z_t\} \cap \{\forall t \in [0,T] \colon Z^\ast_t \in I^\ast\}) = \PP( \{\forall t \in [0,T] \colon Z^\ast_t \in I^\ast\}) .
	\]
	Since $Z = G^\ast(X)$, $Z^\ast = G^\ast(X^\ast)$ and $G^\ast$ is a bi-Lipschitz continuous function with an absolutely continuous derivative, the claim follows with an application of the It\^{o} formula, see e.g.~\cite[Problem 3.7.3]{ks91}.
	
\end{proof}

\section*{Acknowledgement}
I would like to thank Łukasz Stepien for the suggestion to investigate global errors with the coupling of noise technique.

Moreover, I want to express my gratitude to Thomas M\"uller-Gronbach and also Larisa Yaroslavtseva for their encouragement and useful critiques of this article.

\bibliographystyle{acm}
\bibliography{bibfile}

\def\cprime{$'$} \def\cprime{$'$}
\begin{thebibliography}{10}

\bibitem{DGL22}
{\sc Dareiotis, K., Gerencs\'{e}r, M., and L\^{e}, K.}
\newblock Quantifying a convergence theorem of {G}y\"{o}ngy and {K}rylov.
\newblock {\em Ann. Appl. Probab. 33}, 3 (2023), 2291--2323.

\bibitem{ELL24}
{\sc Ellinger, S.}
\newblock Sharp lower error bounds for strong approximation of {SDE}s with
  piecewise {L}ipschitz continuous drift coefficient.
\newblock {\em Journal of Complexity 81\/} (2024), Paper No. 101822, 29.

\bibitem{ELL25Density}
{\sc Ellinger, S.}
\newblock Regularity properties of densities of {SDE}s using the {F}ourier
  analytic approach.
\newblock {\em arXiv:2504.21516\/} (2025).

\bibitem{EMGY25Hoelder}
{\sc Ellinger, S., Müller-Gronbach, T., and Yaroslavtseva, L.}
\newblock On optimal error rates for strong approximation of {SDE}s with a
  {H}\"older continuous drift coefficient.
\newblock {\em arXiv:2504.20728\/} (2025).

\bibitem{EMGY24Sobolev}
{\sc Ellinger, S., M{\"u}ller-Gronbach, T., and Yaroslavtseva, L.}
\newblock On optimal error rates for strong approximation of {SDE}s with a
  drift coefficient of {S}obolev regularity.
\newblock {\em To appear in Ann. Inst. Henri Poincar\'{e} Probab. Stat.\/}
  (2025).

\bibitem{GJS17}
{\sc Gerencs\'{e}r, M., Jentzen, A., and Salimova, D.}
\newblock On stochastic differential equations with arbitrarily slow
  convergence rates for strong approximation in two space dimensions.
\newblock {\em Proc. Roy. Soc. A 473\/} (2017), Paper No. 20170104, 16.

\bibitem{GLN17}
{\sc G{\"o}ttlich, S., Lux, K., and Neuenkirch, A.}
\newblock The {E}uler scheme for stochastic differential equations with
  discontinuous drift coefficient: {A} numerical study of the convergence rate.
\newblock {\em Adv. Difference Equ.\/} (2019), Paper No. 429, 21.

\bibitem{g98b}
{\sc Gy{\"o}ngy, I.}
\newblock A note on {E}uler's approximations.
\newblock {\em Potential Anal. 8}, 3 (1998), 205--216.

\bibitem{gk96b}
{\sc Gy{\"o}ngy, I., and Krylov, N.}
\newblock Existence of strong solutions for {I}t{\^o}'s stochastic equations
  via approximations.
\newblock {\em Probab. Theory Related Fields 105}, 2 (1996), 143--158.

\bibitem{hhj12}
{\sc Hairer, M., Hutzenthaler, M., and Jentzen, A.}
\newblock Loss of regularity for {K}olmogorov equations.
\newblock {\em Ann. Probab. 43}, 2 (2015), 468--527.

\bibitem{HalidiasKloeden2008}
{\sc Halidias, N., and Kloeden, P.~E.}
\newblock A note on the {E}uler-{M}aruyama scheme for stochastic differential
  equations with a discontinuous monotone drift coefficient.
\newblock {\em BIT 48}, 1 (2008), 51--59.

\bibitem{hhmg2019}
{\sc Hefter, M., Herzwurm, A., and M{\"u}ller-Gronbach, T.}
\newblock Lower error bounds for strong approximation of scalar {SDE}s with
  non-{L}ipschitzian coefficients.
\newblock {\em Ann. Appl. Probab. 29}, 1 (2019), 178--216.

\bibitem{hjk11}
{\sc Hutzenthaler, M., Jentzen, A., and Kloeden, P.~E.}
\newblock Strong and weak divergence in finite time of {E}uler's method for
  stochastic differential equations with non-globally {L}ipschitz continuous
  coefficients.
\newblock {\em Proc. R. Soc. Lond. Ser. A Math. Phys. Eng. Sci. 467\/} (2011),
  1563--1576.

\bibitem{JMGY15}
{\sc Jentzen, A., M\"uller-Gronbach, T., and Yaroslavtseva, L.}
\newblock On stochastic differential equations with arbitrary slow convergence
  rates for strong approximation.
\newblock {\em Commun. Math. Sci. 14}, 7 (2016), 1477--1500.

\bibitem{Ka96}
{\sc Kallenberg, O.}
\newblock On the existence of universal functional solutions to classical
  {SDE}'s.
\newblock {\em Ann. Probab. 24\/} (1996), 196--205.

\bibitem{ks91}
{\sc Karatzas, I., and Shreve, S.~E.}
\newblock {\em Brownian motion and stochastic calculus}, second~ed., vol.~113
  of {\em Graduate Texts in Mathematics}.
\newblock Springer-Verlag, New York, 1991.

\bibitem{LS16}
{\sc Leobacher, G., and Sz{\"o}lgyenyi, M.}
\newblock A numerical method for {SDE}s with discontinuous drift.
\newblock {\em BIT 56}, 1 (2016), 151--162.

\bibitem{LS15b}
{\sc Leobacher, G., and Sz{\"o}lgyenyi, M.}
\newblock A strong order 1/2 method for multidimensional {SDE}s with
  discontinuous drift.
\newblock {\em Ann. Appl. Probab. 27\/} (2017), 2383--2418.

\bibitem{LS18}
{\sc Leobacher, G., and Sz\"olgyenyi, M.}
\newblock Convergence of the {E}uler-{M}aruyama method for multidimensional
  {SDE}s with discontinuous drift and degenerate diffusion coefficient.
\newblock {\em Numer. Math. 138}, 1 (2018), 219--239.

\bibitem{MGRY2018}
{\sc M{\"u}ller-Gronbach, T., and Yaroslavtseva, L.}
\newblock A note on strong approximation of {SDE}s with smooth coefficients
  that have at most linearly growing derivatives.
\newblock {\em J. Math. Anal. Appl. 467\/} (2018), 1013--1031.

\bibitem{MGY20}
{\sc M\"{u}ller-Gronbach, T., and Yaroslavtseva, L.}
\newblock On the performance of the {E}uler-{M}aruyama scheme for {SDE}s with
  discontinuous drift coefficient.
\newblock {\em Annales de l'Institut Henri Poincaré (B) Probability and
  Statistics 56}, 2 (2020), 1162--1178.

\bibitem{MGY19b}
{\sc M\"{u}ller-Gronbach, T., and Yaroslavtseva, L.}
\newblock A strong order 3/4 method for {SDE}s with discontinuous drift
  coefficient.
\newblock {\em IMA J. Numer. Anal. 42}, 1 (2022), 229--259.

\bibitem{MGY23}
{\sc M\"{u}ller-Gronbach, T., and Yaroslavtseva, L.}
\newblock Sharp lower error bounds for strong approximation of {SDE}s with
  discontinuous drift coefficient by coupling of noise.
\newblock {\em Ann. Appl. Probab. 33\/} (2023), 902--935.

\bibitem{NSS19}
{\sc Neuenkirch, A., Sz\"olgyenyi, M., and Szpruch, L.}
\newblock An adaptive {E}uler-{M}aruyama scheme for stochastic differential
  equations with discontinuous drift and its convergence analysis.
\newblock {\em {SIAM J. Numer. Anal.} 57\/} (2019), 378--403.

\bibitem{Tag16}
{\sc Ngo, H.-L., and Taguchi, D.}
\newblock Strong rate of convergence for the {E}uler-{M}aruyama approximation
  of stochastic differential equations with irregular coefficients.
\newblock {\em Math. Comp. 85}, 300 (2016), 1793--1819.

\bibitem{Tag2017b}
{\sc Ngo, H.-L., and Taguchi, D.}
\newblock On the {E}uler-{M}aruyama approximation for one-dimensional
  stochastic differential equations with irregular coefficients.
\newblock {\em IMA J. Numer. Anal. 37}, 4 (2017), 1864--1883.

\bibitem{Tag2017a}
{\sc Ngo, H.-L., and Taguchi, D.}
\newblock Strong convergence for the {E}uler-{M}aruyama approximation of
  stochastic differential equations with discontinuous coefficients.
\newblock {\em Statist. Probab. Lett. 125\/} (2017), 55--63.

\bibitem{Pakkanen2010}
{\sc Pakkanen, M.~S.}
\newblock Stochastic integrals and conditional full support.
\newblock {\em Journal of Applied Probability 47}, 3 (2010), 650--667.

\bibitem{RevuzYor1999}
{\sc Revuz, D., and Yor, M.}
\newblock {\em Continuous martingales and {B}rownian motion}, third~ed.,
  vol.~293 of {\em Grundlehren der mathematischen Wissenschaften [Fundamental
  Principles of Mathematical Sciences]}.
\newblock Springer-Verlag, Berlin, 1999.

\bibitem{Y17}
{\sc Yaroslavtseva, L.}
\newblock On non-polynomial lower error bounds for adaptive strong
  approximation of {SDE}s.
\newblock {\em J. Complexity 42\/} (2017), 1--18.

\bibitem{Y2022}
{\sc Yaroslavtseva, L.}
\newblock An adaptive strong order 1 method for {SDE}s with discontinuous drift
  coefficient.
\newblock {\em J. Math. Anal. Appl. 513\/} (2022), Paper No. 126180, 29.

\end{thebibliography}

\end{document}